\documentclass[a4paper,12pt]{article}
\usepackage{url}
\usepackage{amsmath, amsthm}
\usepackage{color}
\usepackage{amsfonts}
\usepackage{graphicx}
\usepackage{nameref}
\usepackage{varioref}
\usepackage{enumitem}
\usepackage{cleveref}
\usepackage{verbatim}
\usepackage{listings}
\usepackage{textcomp}
\usepackage{xcolor}
\usepackage[framed]{matlab-prettifier}
\usepackage[centerlast,small,sc]{caption}
\usepackage{subcaption}
\usepackage[super]{nth}
\usepackage[define-L-C-R]{nicematrix}
\usepackage{booktabs}
\usepackage{multirow}
\usepackage{float}
\usepackage[ruled,vlined]{algorithm2e}
\usepackage{MnSymbol}
\usepackage{hyphenat}
\usepackage{bigdelim}
\usepackage{afterpage}
\DeclareFontFamily{OT1}{pzc}{}
\DeclareFontShape{OT1}{pzc}{m}{it}{<-> s * [1.10] pzcmi7t}{}
\DeclareMathAlphabet{\mathpzc}{OT1}{pzc}{m}{it}

\definecolor{mygray}{gray}{0.4}

\newcommand{\vxb}{\bar{v}^\times}
\newcommand{\vmb}{\bar{v}^-}
\newcommand{\vpb}{{\bar{v}}^{+}}
\newcommand{\vi}{v^1}
\newcommand{\vii}{v^2}
\newcommand{\vp}{v^{+}}
\newcommand{\Qal}{Q^{al}}
\newcommand{\wal}{w^{al}}
\newcommand{\Qo}{Q^{0}}
\newcommand{\vo}{v^{0}}
\newcommand{\Bco}{B_c^{0}}
\newcommand{\Qalo}{Q^{al,0}}
\newcommand{\algorithmQ}{Algorithm~\textbf{Q}}
\newcommand{\algorithmQAL}{Algorithm~\textbf{QAL}}
\newcommand{\algorithmQo}{Algorithm~\textbf{Q0}}
\newcommand{\algorithmQodir}{Algorithm~\textbf{Q0-dir}}
\newcommand{\algorithmQosym}{Algorithm~\textbf{Q0-sym}}
\newcommand{\algorithmQALo}{Algorithm~\textbf{QAL0}}
\newcommand{\algorithmQALodir}{Algorithm~\textbf{QAL0-dir}}
\newcommand{\algorithmQALosym}{Algorithm~\textbf{QAL0-sym}}

\newtheorem{prop}{Proposition}

\newtheorem{theorem}{Theorem}[section]

\newtheorem{definition}{Definition}[section]

\begin{document}

\fontsize{12pt}{12pt}
\usefont{OT1}{cmr}{m}{n}
\begin{center}
	\textbf{On multi-conditioned conic fitting in Geometric
		algebra for conics}
\end{center}

\begin{center}
	Pavel Lou\v cka and Petr Va\v s\'ik
\end{center}

\noindent
\textbf{Abstract.} We introduce several modifications of conic fitting in Geometric algebra for conics by incorporating additional conditions into the optimisation problem. Each of these extra conditions ensure additional geometric properties of a fitted conic, in particular, centre point position at the origin of coordinate system, axial alignment with coordinate axes, or, eventually, combination of both. All derived algorithms are accompanied by a discussion of the underlying algebra and optimisation issues, together with the implementation in MATLAB. Finally, we present examples on a sample dataset and offer possible use of the algorithms.

\vspace*{12pt} 

\section{Introduction}

In this paper we provide modifications to the conic fitting algorithm introduced in \cite{GACFIT}, i.e. an algorithm based on conic representation in Geometric Algebra for Conics (GAC), \cite{GAC}.  Particularly, besides the normalisation condition natural for GAC that was used in the original algorithm, we impose additional geometric conditions on a fitted conic, more specifically, the following: the centre point's position at the origin of coordinate system, axial alignment with coordinate axes, and, eventually, the combination of both. 
The only similar modification of the conic fitting algorithms known to the authors may be found in \cite{stred_elipsy_na_primce}, where the conic is fitted in such a way that its centre lies on a prescribed line. In contrast with this approach, we obtain our solution explicitly as a direct result of the optimisation problem, thanks to its geometric formulation in GAC. 

For a detailed survey of standard algorithms we refer to \cite{fit_alg}. Although all classical algorithms use a linear or quadratic constraint, each of them can be linearised in the sense that the best fit may be found using an eigenvector of a matrix operator. Yet this is the only similarity of the classical algorithms to our approach. 

We recall the original GAC fitting algorithm in Section~\ref{sect3}, where we also present its modification which decreases computational demands. Consequently, in Section~\ref{sect4} we introduce additional geometric conditions in terms of GAC and modify the original algorithm accordingly. We also discuss the optimisation issues w.r.t. the geometric algebra representation and include implementation of the modified algorithms in MATLAB code. In the end of the section, we also offer the overview of the described algorithms. In Section \ref{sect5} we provide an analysis and visualisation of the testing examples and finally, we mention a possible use of the algorithms.

\section{Geometric algebra}
By a geometric algebra (GA) we mean a Clifford algebra with a specific embedding of an Euclidean space (of arbitrary dimension) in such a way that the predefined geometric primitives as well as their transformations are viewed as the algebra elements, precisely multivectors. This concept has been introduced by D. Hestenes in \cite{hestenes} and has been used in many mathematical and engineering applications since, see e.g. \cite{roura}. 

Great computational advantage of GA is that the geometric operations such as intersections, tangents, distances etc. are linear functions and therefore their calculation is efficient. To demonstrate this, we refer to \cite{Perwass} for the basics of geometric algebras, especially for conformal representation of a Euclidean space. Indeed, 3-dimensional Euclidean space is represented in Clifford algebra $\mathcal Cl(4,1)$ with the basis vectors  $\{e_1,e_2,e_3,e_4\},$  and the consequent geometric algebra is often denoted as $\mathbb{G}_{4,1}$ with spheres of all types as geometric primitives and Euclidean transformations at hand, see eg. \cite{dorst}. 

We use the algebra for conics, proposed by C. Perwass to generalise the concept of (2-dimensional) conformal geometric algebra $\mathbb{G}_{3,1}$, \cite{Perwass}, in the sequel referred to as CRA (Compass Ruler Algebra), \cite{dh2}. Let us stress that we use the notation of \cite{GACFIT}. In the usual basis $\bar{n},e_1,e_2,n$, an embedding of a plane in $\mathbb{G}_{3,1}$ is given by
$$
(x,y)\mapsto \bar{n} + xe_1+ye_2+\frac12(x^2+y^2)n,
$$
where $e_1,e_2$ form a Euclidean basis and $\bar n$ and $n$ stand for a specific linear combination of additional basis vectors $e_3, e_4$ with $e_3^2=1$ and $e_4^2=-1$, giving them the meaning of the coordinate origin and infinity, respectively, \cite{Perwass}. Hence the objects representable by vectors in $\mathbb{G}_{3,1}$ are linear combinations of $1, x, y, x^2+y^2$, i.e. circles, lines, point pairs and points. If we want to cover also general conics, we need to add two terms: $\frac12(x^2-y^2)$ and $xy$. It turns out that we need two new infinities for that and also their two corresponding counterparts (Witt pairs), \cite{Lounesto}. Thus the resulting dimension of the space generating the appropriate geometric algebra is eight.

Analogously to CGA and to the notation in \cite{Perwass}, we denote the corresponding basis elements as follows
$$\bar{n}_+,\bar{n}_-,\bar{n}_\times,e_1,e_2,n_+,n_-,n_\times.$$
This notation suggests that the basis elements $e_1,e_2$ will play the usual role of the standard basis of the plane while the null vectors $\bar{n}$, $n$ will represent either  the origin or the infinity. Note that there are three orthogonal `origins' $\bar{n}$ and three corresponding orthogonal `infinities' $n$.
In terms of this basis,  a point of the plane ${\bf x} \in\mathbb{R}^2$ defined by ${\bf x}=xe_1+ye_2$ is embedded using the operator $C: \mathbb{R}^2 \to \mathcal{C}\hspace{-1.5pt}\mathpzc{one}\subset \mathbb{R}^{5,3}$, which is defined by
\begin{align} \label{embedding}
C(x,y)=\bar{n}_+ + xe_1+ye_2+\frac12(x^2+y^2)n_+ + \frac12(x^2-y^2)n_- + xyn_\times,
\end{align}
where the image $\mathcal{C}\hspace{-1.5pt}\mathpzc{one}$ of the plane in $\mathbb{R}^{5,3}$ is an analogue of the conformal cone. In fact, it is a two-dimensional real projective variety determined by five homogeneous polynomials of degree one and two.
\begin{definition}
	Geometric Algebra for Conics (GAC) is the Clifford algebra $\mathbb G_{5,3}$ together with the embedding $\mathbb{R}^2\rightarrow\mathbb{R}^{5,3}$.
\end{definition}
Note that, up to the last two terms, the embedding \eqref{embedding} is the embedding of the plane into the two-dimensional conformal geometric algebra $\mathbb{G}_{3,1}$. In particular, it is evident that the scalar product of two embedded points is the same as in $\mathbb{G}_{3,1}$, i.e. for two points ${\bf x},{\bf y}\in\mathbb{R}^2$ we have
\begin{align*} 
C({\bf x})\cdot C({\bf y})=-\frac12 \|{\bf x}-{\bf y}\|^2,
\end{align*}
where the standard Euclidean norm is considered on the right hand side. This demonstrates the linearisation of the distance problems. In particular, each point is represented by a null vector.
Let us recall that the invertible algebra elements are called versors and they form a group, the Clifford group, and that conjugations with versors give the transformations intrinsic to the  algebra. Namely, if the conjugation with a $\mathbb{G}_{5,3}$ versor $R$ preserves the `cone' $\mathcal{C}\hspace{-1.5pt}\mathpzc{one}$, i.e. for each ${\bf x}\in\mathbb{R}^2$ there exists such a point $\bar{{\bf x}}\in\mathbb{R}^2$ that
\begin{equation*}  
	RC({\bf x})\tilde{R} = C( \bar{{\bf x}}),
\end{equation*}
where $\tilde{R}$ is the reverse of $R$, then ${\bf x}\mapsto\bar{{\bf x}}$ induces a transformation $\mathbb{R}^2\to\mathbb{R}^2$ which is intrinsic to GAC. See \cite{GAC} to find that the conformal transformations are intrinsic to GAC.

Let us also recall the outer (wedge) product, inner product and the duality $A^*=AI^{-1}$.
Henceforth we use the usual definitions as in \cite{Perwass}. Note that in GAC the pseudoscalar (the highest grade element) is given by
$
I=\bar{n}_+\bar{n}_-\bar{n}_\times e_1e_2n_+n_-n_\times.
$

Let us also recall that if a conic $C$ is seen as a wedge of five different points (which determines a conic uniquely), we call the appropriate 5-vector $E^*$ an outer product null space representation (OPNS) and its dual $E$, indeed a one vector, the inner product null space (IPNS) representation. The reason is that if a point $P$ lies on a conic $C$ then
$$P\cdot E=0\quad \text{and}\quad P\wedge E^*=0.$$
Consequently, intersections of two geometric primitives are given as the wedge product of their IPNS representations, i.e. 
$$C_1\cap C_2=E_1\wedge E_2$$
for two conics $C_1,C_2$ and their IPNS representations $E_1$ and $E_2$, respectively, see \cite{GAC}.
It is well known that the type of a given unknown conic can be read off its matrix representation, which in our case for a conic given by vector \eqref{IPNS_conic}  reads
\begin{align} \label{conic_matrix} 
Q = 
\renewcommand{\arraystretch}{1.5}
\setlength\arraycolsep{2pt}
\begin{pmatrix}-\tfrac12(\bar{v}^++\bar{v}^-) & -\tfrac12\bar{v}^\times&\phantom{-}\frac12v^1\\
-\tfrac12\bar{v}^\times & -\frac12(\bar{v}^+-\bar{v}^-) &\phantom{-}\frac12v^2\\
\phantom{-}\frac12v^1 & \phantom{-}\frac12v^2 & \phantom{|}-v^+\end{pmatrix}.
\end{align} 
The entries of \eqref{conic_matrix} can be easily computed by means of the inner product:
\begin{align*}  
q_{11}&=Q_I\cdot\tfrac12(n_+-n_-), \\
q_{22}&=Q_I\cdot\tfrac12(n_++n_-),  \\
q_{33}&=Q_I\cdot \bar{n}_+, \\
q_{12}&=q_{21}=Q_I\cdot\tfrac12 n_\times, \\
q_{13}&=q_{31}=Q_I\cdot\tfrac12 e_1, \\
q_{23}&=q_{32}=Q_I\cdot\tfrac12 e_2. 
\end{align*} 
It is also well known how to determine the internal parameters of an unknown conic, its position and the orientation in the plane from the matrix \eqref{conic_matrix}. Hence all this can be determined from the GAC vector $Q_I$ by means of the inner product.

A general vector in the conic space $\mathbb{R}^{5,3}$ in terms of our basis is of the form $$v=\bar{v}^+\bar{n}_++\bar{v}^-\bar{n}_-+\bar{v}^\times\bar{n}_\times+v^1e_1+v^2e_2+v^+n_++v^-n_-+v^\times n_\times$$
and its inner product with an embedded point is then given by
$$
C(x,y)\cdot v=-\frac12(\bar{v}^++\bar{v}^-)x^2-\bar{v}^\times xy-\frac12(\bar{v}^+-\bar{v}^-)y^2+v^1x+v^2y-v^+,
$$
i.e. by a general polynomial of degree two. Thus the objects representable in GAC are exactly conics.
We also see that the two-dimensional subspace generated by infinities $n_-,n_\times$ is orthogonal to all embedded points. Hence a conic is uniquely represented (in  homogeneous sense) by a vector in $\mathbb{R}^{5,3}$ modulo this subspace. This gives the desired dimension six. In other words, the inner representation of a conic in GAC can be defined as a vector 
\begin{align} \label{IPNS_conic} 
Q_I=\bar{v}^+\bar{n}_++\bar{v}^-\bar{n}_-+\bar{v}^\times\bar{n}_\times+v^1e_1+v^2e_2+v^+n_+.
\end{align} 

The classification of conics is well known. The non-degenerate conics are of three types, the ellipse, hyperbola, and parabola. Let us recall just the form of an ellipse in GAC. An ellipse $E$ with the semi-axes $a,b$ centred in $(u,v)\in\mathbb{R}^2$ rotated by angle $\theta$ is in the GAC inner representation given by
\begin{align*}
	E_I&=\bar{n}_+-(\alpha\cos2\theta)\bar{n}_- -(\alpha\sin2\theta)\bar{n}_\times \\ \nonumber
	&\quad + (u+u\alpha\cos2\theta-v\alpha\sin2\theta)e_1
	+(v+v\alpha\cos2\theta-u\alpha\sin2\theta)e_2 \\ \nonumber
	&\quad +\tfrac12\left(u^2+v^2-\beta-(u^2-v^2)\alpha\cos2\theta-2uv\alpha\sin2\theta\right)n_+.
\end{align*}  
An axes-aligned ellipse $E_I$  with semi-axes $a,b$ and centre in the coordinate system origin is a vector of the form
$$E_I = (a^2+b^2)\bar{n}_++(a^2-b^2)\bar{n}_- -a^2b^2 n_+.$$

\noindent
For proofs and further details see \cite{GACFIT}.

Altogether, GAC is constituted as a Clifford algebra $ \mathcal{C}l(5,3) $ with an associated bilinear form given by a matrix
\begin{align} \label{B} 
\renewcommand{\arraystretch}{1.2}
\setlength\arraycolsep{3pt}
B=\begin{pmatrix}
\phantom{-}0&\phantom{-}0&\phantom{-}0&\phantom{-}0&\phantom{-} 0&\phantom{-}0&\phantom{-}0&-1\phantom{|}\\
\phantom{-}0&\phantom{-}0&\phantom{-}0&\phantom{-}0&\phantom{-} 0&\phantom{-}0&-1&\phantom{-}0\phantom{|}\\
\phantom{-}0&\phantom{-}0&\phantom{-}0&\phantom{-}0&\phantom{-} 0&-1&\phantom{-}0&\phantom{-}0\phantom{|}\\
\phantom{-}0&\phantom{-}0&\phantom{-}0&\phantom{-}1&\phantom{-} 0&\phantom{-}0&\phantom{-}0&\phantom{-}0\phantom{|}\\
\phantom{-}0&\phantom{-}0&\phantom{-}0&\phantom{-}0&\phantom{-}1&\phantom{-}0&\phantom{-}0&\phantom{-}0\phantom{|}\\
\phantom{-}0&\phantom{-}0&-1&\phantom{-}0&\phantom{-} 0&\phantom{-}0&\phantom{-}0&\phantom{-}0\phantom{|}\\
\phantom{-}0&-1&\phantom{-}0&\phantom{-}0&\phantom{-} 0&\phantom{-}0&\phantom{-}0&\phantom{-}0\phantom{|}\\
-1&\phantom{-}0&\phantom{-}0&\phantom{-}0&\phantom{-} 0&\phantom{-}0&\phantom{-}0&\phantom{-}0\phantom{|}\\
\end{pmatrix},
\end{align} 
and an embedded GAC point \eqref{embedding} has the form
\begin{align}\label{GACpoint}
P=&
\setlength\arraycolsep{4pt}
\begin{pmatrix}
0 &0& 1& x& y& \frac12(x^2+y^2)& \frac12(x^2-y^2)& xy
\end{pmatrix}^T,
\end{align}
while the IPNS representation of a general conic section reads
\begin{align}\label{GACconic}
Q_I=&
\setlength\arraycolsep{4pt}
\begin{pmatrix}
\bar{v}^\times&\bar{v}^-&\bar{v}^+&v^1&v^2&v^+&0&0
\end{pmatrix}^T.
\end{align}
The advantage of this notation is that the $4\times4$ matrix in the middle of \eqref{B} is the usual matrix for the inner product in the conformal geometric algebra. Consequently, CRA objects in GAC are built of vectors  with vanishing  first two and last two coordinates and thus they are easier to recognise.

\section{Conic fitting in GAC - original algorithm} \label{sect3}
Having GAC in hand, where conic sections are the intrinsic geometric primitives, we can formulate a fitting algorithm for the conics analogous to the algorithm of fitting circles in CRA. Namely, for a conic represented by a vector $Q$ of the form \eqref{GACconic} and for given points represented by vectors $P_i$ of the form \eqref{GACpoint}, we assume the objective function (cost function) to be given by  
\begin{align}\label{cost}
Q\mapsto\sum_i(P_i\cdot Q)^2,
\end{align}
where we used the inner product between vectors in GAC.
The best conic fitting the points w.r.t. this function is represented by the $Q$ that minimises this function. The geometrically meaningless minimum $Q=0$ is not of the interest and thus the authors of \cite{GACFIT} consider the natural geometric constraint 
\begin{align} \label{constraint}
Q^2=1.
\end{align}
Then the objective function is zero if and only if $P_i\cdot Q=0$ for all $i$, i.e. all points lie on the conic $ Q $.
Note that the square in \eqref{constraint} is taken w.r.t. the geometric product which is the same as the inner product in this case. Both the objective function and the constraint are defined in a covariant way in terms of GAC. The objective function is obviously invariant under all conformal transformations while the constraint is invariant under rotations and scaling only. As argued in \cite{GACFIT}, it is not invariant under translations. Hence the solution of the optimisation problem minimising \eqref{cost} subject to \eqref{constraint} enjoys the same invariant properties by definition. 

This formulation of the problem of fitting conics to points in terms of GAC gives not only a geometric insight but also allows an effective computation of the solution as an eigenvector of a $4\times 4$ matrix which is easy to implement, see Proposition~\ref{solution} below. We use the matrix notation from now on since it is convenient for the implementation. Using the matrix of bilinear form \eqref{B}, the objective function \eqref{cost} reads
\begin{align*}
Q\mapsto\sum_i(P_iBQ)^2=\sum_iQ^TBP_iP_i^TBQ=Q^TPQ,
\end{align*}
and thus it is a quadratic form on $\mathbb{R}^{5,3}$ with the matrix  
\begin{align} \label{P}
P=\sum_i BP_iP_i^TB.
\end{align}
To formulate the solution of our optimisation problem we need to decompose this matrix into blocks. Note that the points in GAC lie in a six-dimensional subspace and thus the rank of the matrix is six. More precisely, it has the form  
\begin{align}
P=
\renewcommand{\arraystretch}{1.2}
\setlength\arraycolsep{3pt}
\begin{pmatrix}\label{Pblocks}
P_0 & P_1 & 0\\
P_1^T & P_c & 0\\
0 &0 &0
\end{pmatrix},
\end{align}
where $P_0$ is a $2\times 2$ matrix, $P_1$ is a $2\times 4$ matrix and $P_c$ is a $4 \times 4$ matrix. The subscript $c$ denotes that this block corresponds to the CRA part in GAC. Similarly, $B_c$ denotes the middle part of $B$ and it coincides with the matrix of the inner product in CRA, i.e.
\begin{align} \label{Bc}
	B_c = 
	\renewcommand{\arraystretch}{1.2}
	\setlength\arraycolsep{3pt}
	\begin{pmatrix}
		\phantom{-}0 &\phantom{-}0 &\phantom{-}0 & -1\phantom{|} \\
		\phantom{-}0 &\phantom{-} 1 &\phantom{-} 0 &\phantom{-} 0\phantom{|} \\
		\phantom{-}0 &\phantom{-} 0 &\phantom{-} 1 &\phantom{-} 0\phantom{|} \\
		-1 &\phantom{-} 0 &\phantom{-} 0 &\phantom{-} 0\phantom{|} \\ 
	\end{pmatrix}.
\end{align}
Thus, we are ready to formulate the following Proposition, \cite{GACFIT}.
\begin{prop} \label{solution}
	The solution to the optimisation problem \eqref{cost}, \eqref{constraint} for conic fitting in GAC is given by
	$Q=\setlength\arraycolsep{3pt}\begin{pmatrix}w &v& 0\end{pmatrix}^T,$ where $v$ is a 4-dimensional eigenvector corresponding to the minimal non-negative eigenvalue of 
	\begin{align} \label{Pcon}
	P_{con}=B_c(P_c-P_1^TP^{-1}_0P_1)
	\end{align}
	and $w$ is a 2-dimensional vector of the form
	\begin{align} \label{w}
	w=-P_0^{-1}P_1v.
	\end{align}
\end{prop}
\begin{proof}
	Note that the inversion in \eqref{w} exists up to the case when all points lie on a single line. This singular case must be detected and treated separately. We solve the optimisation problem by the method of Lagrange multipliers.
	The constraint \eqref{constraint} of the optimisation problem reads $Q^TBQ=1$ in matrix form. Hence the Lagrange function is
	$$
	P_\lambda(Q):=Q^TPQ+\lambda (1-Q^TBQ).
	$$
	We find its stationary points by differentiation w.r.t. unknown vector $Q$ and we use the block form of $P$ given by \eqref{Pblocks}. Using the block decomposition of $Q$ from the proposition the Lagrange function reads
	$$
	P_\lambda(Q)=w^TP_0w+2w^TP_1v+v^TP_cv+\lambda(1-v^TB_cv),
	$$
	and its stationary points satisfy the following system of linear equations
	\begin{align*}
	0=\frac{\partial P_\lambda}{\partial w}&=2(P_0w+P_1v),
	\\
	0=\frac{\partial P_\lambda}{\partial v}&=2(P_1^Tw+P_cv-\lambda B_cv),
	\\
	0=\frac{\partial P_\lambda}{\partial \lambda}&=1-v^TB_cv.
	\end{align*}
	From the first equation we directly get \eqref{w} and the substitution of the corresponding term in the second equation yields the equation for generalised eigenvectors 
	\begin{align*} \label{eigenvectors}
	(P_c-P_1^TP_0^{-1}P_1)v=\lambda B_c v.
	\end{align*}
	Since $B_c$ squares to one, it can be reduced to the equation for standard eigenvectors by multiplying it by $B_c$ from left. Hence $v$-part of each stationary point of our Lagrange function is an eigenvalue of operator \eqref{Pcon}. Moreover, it must be normalised according to the third equation of the system above 
	\begin{equation}
	\label{hvezda}
	v^TB_cv=1.
	\end{equation}
	
	Note that this is not always possible since $B_c$ is not positive definite. The only thing left is to find the minimum among the stationary points. Therefore, we compute the values of the objective function \eqref{cost}. Let $v_\lambda$ denote the eigenvector corresponding to eigenvalue $\lambda$ and let $Q_\lambda$ be the corresponding conic $Q_\lambda=(-P_0^{-1}P_1v_\lambda\quad v_\lambda \quad 0)$. By means of the block decomposition of $P$ and the definition of $v_\lambda$ we compute
	$$
	P(Q_\lambda)=Q_\lambda ^T P Q_\lambda=v_\lambda ^T(P_c-P_1^TP_0^{-1}P_1)v_\lambda)\lambda v_\lambda^TB_cv_\lambda=\lambda v_\lambda ^T B_c v_\lambda =\lambda >0,
	$$ 
	assuming that $v_\lambda$ is normalised as in \eqref{hvezda}. The eigenvectors corresponding to negative eigenvalues cannot be normalised in the sense of \eqref{hvezda} and thus they are not stationary points. 
	So the minimum of the objective function is attained if $v$ is the eigenvector corresponding to the least but non-negative eigenvalue of $P_{con}$.
\end{proof}

\subsection*{Implementation}
The original fitting algorithm is very easy to implement due to the matrix formulation of Proposition~\ref{solution}. We summarise the procedure in seven steps, each of them followed by the corresponding MATLAB code. Let us just remark that all the algorithms described in the text use input data in the form of point matrix $ p $, with $ i $-th point $ p_i=(x_i,y_i)^T $ constituting the $ i $-th column, i.e.
\begin{align*}
	p = 
	\setlength\arraycolsep{3pt}
	\begin{pmatrix}
		x_1 & x_2 & \ldots & x_N \\
		y_1 & y_2 & \ldots & y_N 	
	\end{pmatrix}.
\end{align*}
\\
\rule{\linewidth}{1.5pt}
\\
\textbf{\algorithmQ}
\\
\rule[6pt]{\linewidth}{0.75pt} \vspace*{-21pt}
\begin{enumerate}
	\item Definition of matrix $B$ for GAC inner product according to \eqref{B} and its $4\times 4$ submatrix $B_c$.
	\begin{lstlisting}[framexrightmargin=-3pt, style=Matlab-editor]
B = zeros(8);
I3 = [0 0 1; 0 1 0; 1 0 0];
B(1:3,6:8) = -I3;
B(4:5,4:5) = eye(2);
B(6:8,1:3) = -I3;	
Bc = B(3:6,3:6);
	\end{lstlisting}
	\vspace{-0.6em}
	\item Formation of data matrix $D$ of size $8\times N$, where $i$-th  column is the GAC vector $P_i$ representing the $i$-th point from the data set according to \eqref{GACpoint}.
	\begin{lstlisting}[framexrightmargin=-3pt, style=Matlab-editor]
D = zeros(8,N);
D(3,:) = ones(1,N);
D(4:5,:) = p;
D(6,:) = 1/2*(p(1,:).^2 + p(2,:).^2);
D(7,:) = 1/2*(p(1,:).^2 - p(2,:).^2);
D(8,:) = p(1,:).*p(2,:);
	\end{lstlisting}
	\vspace{-0.6em}
	\item Computation of the symmetric matrix $P$ for the objective function by \eqref{P} and the definition of its blocks $P_0$, $P_1$ and $P_c$ according to \eqref{Pblocks}.
	\begin{lstlisting}[framexrightmargin=-3pt, style=Matlab-editor]
P = 1/N*B*(D*D')*B';
Pc = P(3:6,3:6);
P0 = P(1:2,1:2);
P1 = P(1:2,3:6);
	\end{lstlisting}	
	\vspace{-0.6em}
	\item Formation of matrix $P_{con}$ according to \eqref{Pcon} and computation of its eigenvalues and eigenvectors.
	\begin{lstlisting}[framexrightmargin=-3pt, style=Matlab-editor]
Pcon = Bc*(Pc-P1'*(P0\P1));
[EV,ED] = eig(Pcon);
EW = diag(ED);
	\end{lstlisting}	
	\vspace{-0.6em}
	\item Finding the eigenvector $v$ corresponding to the least non-negative eigenvalue of $P_{con}$.
	\begin{lstlisting}[framexrightmargin=-3pt, style=Matlab-editor]
k_opt = find(EW == min(EW(EW>0)));
v_opt = EV(:,k_opt);
	\end{lstlisting}
	\vspace{-0.6em}
	\item Normalisation of the optimal vector $ v $ according to constraint \eqref{hvezda}
	\begin{lstlisting}[framexrightmargin=-3pt, style=Matlab-editor]
kappa = v_opt'*Bc*v_opt;
v_opt = 1/sqrt(kappa)*v_opt; 
	\end{lstlisting}	
	\vspace{-0.6em}
	\item Computation of $w$ by \eqref{w} and forming the optimal vector $Q$ according to Proposition~\ref{solution}.
	\begin{lstlisting}[framexrightmargin=-3pt, style=Matlab-editor]
w = -P0\P1*v_opt;
Q = [w;v_opt;0;0];
	\end{lstlisting}		
\end{enumerate}
\rule[12pt]{\linewidth}{1.5pt}
Let us note that the optimal eigenvector $ v $ obtained in the step 5. of the algorithm is (thanks to MATLAB) already normalised w.r.t. the Euclidean norm, so the normalisation of the vector w.r.t. GAC norm in the step 6. might seem redundant, though in some cases it is necessary to perform it. If one is interested only in conic's representation in the Euclidean plane, then the GAC normalisation is really unnecessary, since multiplying the vector $ v $ by a non-zero constant does not affect it. Indeed, GAC is a projective space and the representatives of conics, i.e. 1-vectors, are elements of the projective cone. Nevertheless, the normalisation may affect intersection and distance calculations, because it changes the value of the inner product.

Let us also remark that the code in the \algorithmQ{} may be further optimised by modification of steps 1. and 2. For example, we can see that the first two rows of the data matrix $ D $ are by definition zero, resulting then in a block of 0's in the matrix $ P $, while this block is no further used in the consequent computation. The zero block in the matrix $ P $ (though redundant) is not that critical itself, but the computation of the matrix $ P $ described in the step 3. uses multiplication of matrix $ D $ with its transpose and thus constitutes both unnecessary element multiplications and storage of useless zero rows of matrix $ D $. Therefore, by removing first two rows from the vector \eqref{GACpoint} we can define the $ i $-th reduced GAC point vector $ \hat{P}_i $ of the form
\begin{equation*} 
	\hat{P}_i = 
	\setlength\arraycolsep{4pt}
	\begin{pmatrix}
		1& x_i& y_i& \frac12(x_i^2+y_i^2)& \frac12(x_i^2-y_i^2)& x_i y_i
	\end{pmatrix}^T
\end{equation*}
and create the corresponding reduced data matrix $ \hat{D} $ of size $ 6 \times N $, where the $ i $-th column is $ \hat{P}_i $. Additionally, by removing the first two columns and the last two rows from the matrix $ B $ of the form \eqref{B}, we also define a reduced matrix for GAC inner product
\begin{align*}
	\hat{B} = 
	\renewcommand{\arraystretch}{1.2}
	\setlength\arraycolsep{3pt}
	\begin{pmatrix}
		\phantom{-}0&\phantom{-}0&\phantom{-} 0&\phantom{-}0&\phantom{-}0&-1\phantom{|}\\
		\phantom{-}0&\phantom{-}0&\phantom{-}0&\phantom{-}0& -1&\phantom{-}0\phantom{|}\\
		\phantom{-}0&\phantom{-}0&\phantom{-} 0&-1&\phantom{-}0&\phantom{-}0\phantom{|}\\
		\phantom{-}0&\phantom{-}1&\phantom{-} 0&\phantom{-}0&\phantom{-}0&\phantom{-}0\phantom{|}\\
		\phantom{-}0&\phantom{-}0&\phantom{-}1&\phantom{-}0&\phantom{-}0&\phantom{-}0\phantom{|}\\
		-1&\phantom{-}0&\phantom{-} 0&\phantom{-}0&\phantom{-}0&\phantom{-}0\phantom{|} 
	\end{pmatrix}.
\end{align*} 
Consequently, step 1. of the algorithm can be replaced by 
\\\\
\begin{minipage}{\textwidth}
	\begin{minipage}{.05\textwidth}
		\lstset{
			showlines=true
		}	
		\begin{lstlisting}[style=Matlab-editor, rulecolor=\color{white}, mathescape]
$1 \normalfont \text{*}$
			
				
	 
	 		
			
		\end{lstlisting}
	\end{minipage}
	\hfill
	\begin{minipage}{.95\textwidth}
		\begin{lstlisting}[framexrightmargin=-5.5pt, style=Matlab-editor]
B = zeros(6);
I3 = [0 0 1;0 1 0; 1 0 0];
B(1:3,4:6) = -I3;
B(4:5,2:3) = eye(2);
B(6,1) = -1;	
Bc = B(3:6,1:4);
		\end{lstlisting}
	\end{minipage}
\end{minipage}
and step 2. by 
\\\\
\begin{minipage}{\textwidth}
	\begin{minipage}{.05\textwidth}
		\lstset{
			showlines=true
		}
		\begin{lstlisting}[style=Matlab-editor, rulecolor=\color{white}, mathescape]
$2 \normalfont \text{*}$



			
		\end{lstlisting}
	\end{minipage}
	\hfill
	\begin{minipage}{.95\textwidth}
		\begin{lstlisting}[framexrightmargin=-5.5pt, style=Matlab-editor]
D = ones(6,N);
D(2:3,:) = p;
D(4,:) = 1/2*(p(1,:).^2 + p(2,:).^2);
D(5,:) = 1/2*(p(1,:).^2 - p(2,:).^2);
D(6,:) = p(1,:).*p(2,:);
		\end{lstlisting}
	\end{minipage}
\end{minipage}	
As a result, the matrix $ P $ computed in the step 3. does not contain the zero blocks as in \eqref{Pblocks} and takes the form of a reduced matrix
\begin{align*}
	\hat{P}=
	\renewcommand{\arraystretch}{1.2}
	\setlength\arraycolsep{3pt}
	\begin{pmatrix}
		P_0 & P_1\\
		P_1^T & P_c\\
	\end{pmatrix}
\end{align*}
of size $ 6 \times 6 $. Moreover, the rest of the algorithm remains unaffected by this change, so the steps 3. to 7. may stay the same. Let us note that we can also analogously omit zeroes in GAC conic vector of the form \eqref{GACconic}, and thus create a reduced conic vector
\begin{align*}
	\hat{Q}=&
	\setlength\arraycolsep{4pt}
	\begin{pmatrix}
		\bar{v}^\times&\bar{v}^-&\bar{v}^+&v^1&v^2&v^+
	\end{pmatrix}^T,
\end{align*}
and then compute the objective function \eqref{cost} equivalently as
\begin{align*}
	\hat{Q} \mapsto \hat{Q}^T \hat{P} \hat{Q}.
\end{align*}

\section{Conic fitting in GAC - additional conditions} \label{sect4}
One of the main features of the above-presented algorithm is that the sought conic is not limited in advance by any prescribed geometric condition. In other words, parameters of the fitted conic such as tilt of its axes or its centre point position in the plane are not known beforehand, being just a result of the optimisation process. 

While this type of fit may be useful in some cases, there are problems demanding a fit with one or more additional geometric conditions constraining the output conic. In this section we focus on forming three types of fitting algorithms, each resulting in either of following conics:
\begin{enumerate}
	\item \textit{conic having its axes aligned with coordinate axes}
	\item \textit{conic having its centre point at the coordinate system origin}
	\item \textit{conic satisfying both} 1. \textit{and} 2.
\end{enumerate}

Before further discussion on the above-mentioned types of conic fitting, let us briefly recall that a conic $ Q_I $ in the IPNS representation of the form \eqref{GACconic} can be as well described as an implicit two variable function
\begin{align} \label{Conic_xy}
	Q(x,y) = \mathcal{A} x^2 + \mathcal{B} xy + \mathcal{C} y^2 + \mathcal{D} x + \mathcal{E} y + \mathcal{F} = 0
\end{align}
with the corresponding matrix 
\begin{align} \label{conic_matrix_xy} 
	Q=
	\renewcommand{\arraystretch}{1.3}
	\setlength\arraycolsep{4pt}
	\begin{pmatrix}
		\mathcal{A} & \mathcal{B}/2& \mathcal{D}/2\\
		\mathcal{B}/2 & \mathcal{C} & \mathcal{E}/2\\
		\mathcal{D}/2 & \mathcal{E}/2 & \mathcal{F}
	\end{pmatrix}.
\end{align} 

Let us just note that the coefficients of \eqref{Conic_xy} are intentionally written in calligraphic style, in order to avoid confusion in notation.

By comparing the matrices \eqref{conic_matrix_xy} and \eqref{conic_matrix} we can easily describe the relationships between the coefficients of $ Q(x,y) $ and the elements of vector $ Q_I $, for example 
\begin{align*}
	 \mathcal{F} =  -\vp.
\end{align*}

Such a direct connection between the IPNS representation and implicit function of a conic will later allow us to easily express the geometric conditions in terms of GAC. As we will see, thanks to this connection and the matrix-vector formulation of the original problem, our conic fitting algorithms with additional conditions can be described analogously to the original one and solved in a way very similar to Proposition~\ref{solution}.

\subsection{Axes-aligned conic}

\begin{theorem} \label{theorem_aligned}
	Conic $ Q $ represented in the form of vector \eqref{GACconic} is axes-aligned if and only if 
	\begin{align*}
		\vxb = 0.
	\end{align*}
\end{theorem}
\begin{proof}
	It is generally known that necessary and sufficient condition for axial alignment of a conic $ Q(x,y) $ in the form \eqref{Conic_xy} is the zero coefficient $ \mathcal{B} $, i.e.
	\begin{align*}
		Q(x,y) \text{ is axes-aligned} \iff \mathcal{B}=0.
	\end{align*}
	By comparing the entries of matrices \eqref{conic_matrix_xy} and \eqref{conic_matrix} we get 
	\begin{align*}
		\vxb = -\mathcal{B},
	\end{align*}
	therefore
	\begin{align*}
		Q \text{ is axes-aligned} \iff \vxb=0.
	\end{align*}
\end{proof}

For more details on the properties of conics, see \cite{handbook} or \cite{geombook}. 

With the axial alignment expressed in terms of GAC, we can substitute the acquired condition into the vector \eqref{GACconic}, yielding a vector IPNS representation of an axes-aligned conic $ \Qal $ of the form
\begin{align}\label{IPNS_conic_aligned}
	\Qal = 
	\setlength\arraycolsep{4pt}
	\begin{pmatrix}
		0 & \vmb & \vpb & \vi & \vii & \vp & 0 & 0
	\end{pmatrix}^T.
\end{align}

Consequently, we may formulate an optimisation problem almost identical to the original problem of minimising objective function \eqref{cost} subject to \eqref{constraint}, however, instead of searching for a general conic $ Q  $ of the form \eqref{GACconic} we will assume the optimal conic to be axes-aligned. Hence, for an axes-aligned conic $ \Qal $ of the form \eqref{IPNS_conic_aligned} and for given points represented by vectors $ P_i $ of the form \eqref{GACpoint}, the objective function reads
\begin{align}\label{cost_aligned}
	\Qal\mapsto\sum_i(P_i\cdot \Qal)^2,
\end{align}
where $ \cdot $ represents the inner product between GAC vectors. As well as in the original fitting problem, we consider the constraint
\begin{align} \label{constraint_aligned}
	{\Qal}^2=1.
\end{align}
Let us stress that the square in \eqref{constraint_aligned} is taken w.r.t. the geometric product which is the same as the inner product in this case.

To solve the upcoming problem we will use again the method of Lagrange multipliers and the block decomposition of the matrix $ P $ of the form \eqref{P}. However, because of the slight change of the resulting conic vector, it is necessary to partition matrix $ P $ in a manner different to \eqref{Pblocks}. Firstly, let us denote the (generally) non-zero entries of symmetric matrix $ P $ as $ p_{ij} $, i.e.
\begin{align} \label{Pentries}
	P = 
	\renewcommand{\arraystretch}{1.2}
	\setlength\arraycolsep{3pt}
	\begin{pmatrix}
		p_{11} & p_{21} & p_{31} & p_{41} & p_{51} & p_{61} & \phantom{i}0 & \phantom{2}0 \\
		p_{21} & p_{22} & p_{32} & p_{42} & p_{52} & p_{62} & \phantom{i}0 & \phantom{2}0 \\
		p_{31} & p_{32} & p_{33} & p_{43} & p_{53} & p_{63} & \phantom{i}0 & \phantom{2}0 \\
		p_{41} & p_{42} & p_{43} & p_{44} & p_{54} & p_{64} & \phantom{i}0 & \phantom{2}0 \\
		p_{51} & p_{52} & p_{53} & p_{54} & p_{55} & p_{65} & \phantom{i}0 & \phantom{2}0 \\
		p_{61} & p_{62} & p_{63} & p_{64} & p_{65} & p_{66} & \phantom{i}0 & \phantom{2}0 \\
		0 & 0 & 0 & 0 & 0 & 0 & \phantom{i}0 & \phantom{2}0 \\
		0 & 0 & 0 & 0 & 0 & 0 & \phantom{i}0 & \phantom{2}0 
	\end{pmatrix}.
\end{align}
The objective function \eqref{cost_aligned} in the matrix notation reads
\begin{align*}
	\Qal \mapsto {\Qal}^T P \Qal
\end{align*}
and the matrix $ P $ is thus a matrix of a quadratic form w.r.t. the vector $ \Qal $. Since the first entry in the vector $ \Qal $ is zero, we are not interested in the first row or the first column of the matrix $ P $. Therefore, the desired decomposition has the form
\begin{align} \label{Pblocks_aligned}
	P = 
	\renewcommand{\arraystretch}{1.3}
	\setlength\arraycolsep{3pt}
	\begin{pNiceArray}{c|c|cccc|cc}
		p_{11} & p_{21} & p_{31} & p_{41} & p_{51} & p_{61} & 0 & 0 \\
		\hline
		p_{21} & \Block{1-1}{P_0^{al}} & \Block{1-4}{P_1^{al}} &  &  &  & 0 & 0 \\
		\hline
		p_{31} & \Block{4-1}{{P_1^{al}}^T } & \Block{4-4}{P_c} &  &  &  & 0 & 0 \\
		p_{41} &  &  &  &  &  & 0 & 0 \\
		p_{51} &  &  &  &  &  & 0 & 0 \\
		p_{61} &  &  &  &  &  & 0 & 0 \\
		\hline
		0 & 0 & 0 & 0 & 0 & 0 & 0 & 0 \\
		0 & 0 & 0 & 0 & 0 & 0 & 0 & 0 
	\end{pNiceArray},
\end{align}
where $ P_0^{al} $ is a real number, $ P_1^{al} $ is a 4-dimensional vector and $ P_c $ is a $ 4 \times 4 $ matrix corresponding to the CRA part in GAC. Let us note that the matrices $ P_0^{al} $, $ P_1^{al} $ and $ {P_1^{al}}^T $ are created as subparts of matrices $ P_0 $, $ P_1 $ and $ {P_1}^T $  from partition \eqref{Pblocks} by omitting the first row and the first column in the matrix $ P $. Submatrix $ P_c $ remains unaffected by this change, so we keep its notation.

\begin{prop}\label{solution_aligned}
	The solution to the optimisation problem \eqref{cost_aligned}, \eqref{constraint_aligned} for conic fitting in GAC is given by
	$\Qal= \setlength\arraycolsep{3pt} \begin{pmatrix} 0 & \wal & v & 0 & 0 \end{pmatrix}^T,$ where $v$ is a 4-dimensional eigenvector corresponding to the minimal non-negative eigenvalue of 
	\begin{align*} 
		P_{con}^{al} = B_c(P_c-{P_1^{al}}^T{P_0^{al}}^{-1}P_1^{al})
	\end{align*}
	and $ \wal \equiv \vmb $ is a real number acquired as
	\begin{align} \label{w_aligned}
		\wal = -{P_0^{al}}^{-1}P_1^{al}v.
	\end{align}
\end{prop}
\begin{proof}
	The inversion in \eqref{w_aligned} exists up to the case when all points lie on a double-line $ x^2 - y^2 = 0 $, so this case must be detected and treated separately. Let us also note that relabelling of the variable $ \vmb $ to $ \wal $ is not necessary at all and has been made only to emphasise the correspondence of this variable to the matrices $ P_0^{al} $, $ P_1^{al} $ and $ {P_1^{al}}^T$, respectively. 
	
	As already said, this approach to the solution is basically the same as in the proof of Proposition~\ref{solution}.
	Using the method of the Lagrange multipliers, we obtain the Lagrange function in the form
	$$
	P_\lambda(\Qal) := {\Qal}^TP\Qal + \lambda (1-{\Qal}^TB\Qal).
	$$
	and after applying the block decomposition \eqref{Pblocks_aligned}, it can be rewritten as
	$$
	P_\lambda(\Qal) = {\wal}^2P_0^{al} + 2{\wal}P_1^{al}v + v^TP_cv+\lambda(1-v^TB_cv).
	$$
	We find its stationary points by differentiation w.r.t. the unknown vector $\Qal$ and by solving the following system of linear equations
	\begin{align*}
		0 = &\frac{\partial P_\lambda}{\partial \wal} = 2(\wal P_0^{al} + P_1^{al}v),
		\\
		0 = &\frac{\partial P_\lambda}{\partial v} = 2(\wal {P_1^{al}}^T + P_cv-\lambda B_cv),
		\\
		0 = &\frac{\partial P_\lambda}{\partial \lambda} = 1-v^TB_cv.
	\end{align*}
	The rest of the proof proceeds analogously to the proof of Proposition~\ref{solution}. 
\end{proof}

\subsection*{Implementation} 

The algorithm for fitting an axes-aligned conic $ \Qal $ has the form very similar to \algorithmQ, containing only a few adjustments of the involved matrices. For the sake of code optimisation we follow the example of \algorithmQ{} and employ the reduced forms of vectors and matrices in the first steps of algorithm. Thus, by removing the first two zeros and the mixed term $ xy $ from the vector \eqref{GACpoint} we define the $ i $-th reduced GAC point vector $ \hat{P}_i^{al} $ of the form
\begin{equation*} 
	\hat{P}_i^{al} = 
	\setlength\arraycolsep{4pt}
	\begin{pmatrix}
		1& x_i& y_i& \frac12(x_i^2+y_i^2)& \frac12(x_i^2-y_i^2)
	\end{pmatrix}^T
\end{equation*}
and create the corresponding reduced data matrix $ \hat{D}^{al} $ of the size $ 5 \times N $, where the $ i $-th column is $ \hat{P}_i^{al} $. Additionally, by removing the rows 1, 7 and 8 and columns 1,2 and 8 from the matrix $ B $ of the form \eqref{B}, we also define a~reduced matrix for GAC inner product
\\
\vspace{-1.2em}
\begin{align*}
	\hat{B}^{al} = 
	\renewcommand{\arraystretch}{1.2}
	\setlength\arraycolsep{3pt}
	\begin{pmatrix}
		\phantom{-}0&\phantom{-}0&\phantom{-}0&\phantom{-}0&-1\phantom{|}\\
		\phantom{-}0&\phantom{-}0&\phantom{-} 0&-1&\phantom{-}0\phantom{|}\\
		\phantom{-}0&\phantom{-}1&\phantom{-} 0&\phantom{-}0&\phantom{-}0\phantom{|}\\
		\phantom{-}0&\phantom{-}0&\phantom{-}1&\phantom{-}0&\phantom{-}0\phantom{|}\\
		-1&\phantom{-}0&\phantom{-} 0&\phantom{-}0&\phantom{-}0\phantom{|} 
	\end{pmatrix},
\end{align*} 
\vspace{-0.8em}
\\
while the submatrix $ B_c $ remains the same as in \eqref{Bc}.
Consequently, the algorithm can be summarised in 7 steps analogous to the original algorithm by the following MATLAB code (individual steps are only numbered, their description would be the same as in \algorithmQ):
\\
\rule{\linewidth}{1.5pt}
\\
\textbf{\algorithmQAL}
\\
\rule[6pt]{\linewidth}{0.75pt} 
\\
\begin{minipage}{\textwidth}
	\begin{minipage}{.05\textwidth}
		\lstset{
			showlines=true
		}
		\begin{lstlisting}[style=Matlab-editor, rulecolor=\color{white}, mathescape]
$1.$





		\end{lstlisting}
	\end{minipage}
	\hfill
	\begin{minipage}{.95\textwidth}
		\begin{lstlisting}[framexrightmargin=-5.5pt, style=Matlab-editor]
B = zeros(5);
I2 = [0 1; 1 0];
B(1:2,4:5) = -I2;
B(3:4,2:3) = eye(2);
B(5,1) = -1;
Bc = B(2:5,1:4);
		\end{lstlisting}
	\end{minipage}
\end{minipage}
\begin{minipage}{\textwidth}
	\begin{minipage}{.05\textwidth}
		\lstset{
			showlines=true
		}
		\begin{lstlisting}[style=Matlab-editor, rulecolor=\color{white}, mathescape]
$2.$



		\end{lstlisting}
	\end{minipage}
	\hfill
	\begin{minipage}{.95\textwidth}
		\begin{lstlisting}[framexrightmargin=-5.5pt, style=Matlab-editor]
D = ones(5,N);
D(2:3,:) = p;
D(4,:) = 1/2*(p(1,:).^2 + p(2,:).^2);
D(5,:) = 1/2*(p(1,:).^2 - p(2,:).^2);
		\end{lstlisting}
	\end{minipage}
\end{minipage}
\begin{minipage}{\textwidth}
	\begin{minipage}{.05\textwidth}
		\lstset{
			showlines=true
		}
		\begin{lstlisting}[style=Matlab-editor, rulecolor=\color{white}, mathescape]
$3.$



		\end{lstlisting}
	\end{minipage}
	\hfill
	\begin{minipage}{.95\textwidth}
		\begin{lstlisting}[framexrightmargin=-5.5pt, style=Matlab-editor]
P = 1/N*B*(D*D')*B';
Pc = P(2:5,2:5);
P0 = P(1,1);
P1 = P(1,2:5);	
		\end{lstlisting}
	\end{minipage}
\end{minipage}
\begin{minipage}{\textwidth}
	\begin{minipage}{.05\textwidth}
		\lstset{
			showlines=true
		}
		\begin{lstlisting}[style=Matlab-editor, rulecolor=\color{white}, mathescape]
$4.$


		\end{lstlisting}
	\end{minipage}
	\hfill
	\begin{minipage}{.95\textwidth}
		\begin{lstlisting}[framexrightmargin=-5.5pt, style=Matlab-editor]
Pcon = Bc*(Pc-P1'*1/P0*P1);
[EV,ED] = eig(Pcon);
EW = diag(ED);		
		\end{lstlisting}
	\end{minipage}
\end{minipage}
\begin{minipage}{\textwidth}
	\begin{minipage}{.05\textwidth}
		\lstset{
			showlines=true
		}
		\begin{lstlisting}[style=Matlab-editor, rulecolor=\color{white}, mathescape]
$5.$
			
		\end{lstlisting}
	\end{minipage}
	\hfill
	\begin{minipage}{.95\textwidth}
		\begin{lstlisting}[framexrightmargin=-5.5pt, style=Matlab-editor]
k_opt = find(EW == min(EW(EW>0)));
v_opt = EV(:,k_opt);			
		\end{lstlisting}
	\end{minipage}
\end{minipage}
\begin{minipage}{\textwidth}
	\begin{minipage}{.05\textwidth}
		\lstset{
			showlines=true
		}
		\begin{lstlisting}[style=Matlab-editor, rulecolor=\color{white}, mathescape]
$6.$
			
		\end{lstlisting}
	\end{minipage}
	\hfill
	\begin{minipage}{.95\textwidth}
		\begin{lstlisting}[framexrightmargin=-5.5pt, style=Matlab-editor]
kappa = v_opt'*Bc*v_opt;
v_opt = 1/sqrt(kappa)*v_opt; 
		\end{lstlisting}
	\end{minipage}
\end{minipage}
\begin{minipage}{\textwidth}
	\begin{minipage}{.05\textwidth}
		\lstset{
			showlines=true
		}
		\begin{lstlisting}[style=Matlab-editor, rulecolor=\color{white}, mathescape]
$7.$
			
		\end{lstlisting}
	\end{minipage}
	\hfill
	\begin{minipage}{.95\textwidth}
		\begin{lstlisting}[framexrightmargin=-5.5pt, style=Matlab-editor]
w = -1/P0*P1*v_opt;
Q = [0;w;v_opt;0;0];
		\end{lstlisting}
	\end{minipage}
\end{minipage}
\vspace{-0.5em}
\\
\rule[12pt]{\linewidth}{1.5pt}
Computation of the matrix $ P $ in step 3. then yields a $ 5 \times 5$ reduced matrix of the form
\begin{align*} 
	\hat{P}^{al}=
	\renewcommand{\arraystretch}{1.6}
	\setlength\arraycolsep{2pt}
	\begin{pmatrix}
		P_0^{al} & P_1^{al}\\
		{P_1^{al}}^T & P_c\\
	\end{pmatrix}.
\end{align*}
Furthermore, by omitting zeroes in GAC vector of an axes-aligned conic of the form \eqref{IPNS_conic_aligned} we define a reduced conic vector
\begin{align*}
	\hat{Q}^{al}=&
	\setlength\arraycolsep{4pt}
	\begin{pmatrix}
		\vmb & \vpb & \vi & \vii & \vp 
	\end{pmatrix}^T,
\end{align*}
and formulate an objective function \eqref{cost_aligned} equivalently as
\begin{align*}
	\hat{Q}^{al} \mapsto \hat{Q}^{{al}^T} \hat{P}^{al} \hat{Q}^{al}.
\end{align*}

\subsection{Origin-centred conic} \label{subsec:4.2}

\begin{theorem} \label{theorem_centred}
	A central conic $ Q $ represented by a vector \eqref{GACconic} has its centre at the origin of the coordinate system if and only if 
	\begin{align*}
		(\vi = 0) \wedge (\vii = 0).
	\end{align*}
\end{theorem}
\begin{proof}
	Firstly, by comparing the entries of matrices \eqref{conic_matrix} and \eqref{conic_matrix_xy} we conclude that
	\begin{align*}
		\vi = \mathcal{D}, \\
		\vii = \mathcal{E},
	\end{align*}
	so the proof can be equivalently written in terms of the function $Q(x,y)$ coefficients in the form  \eqref{Conic_xy}.  We have to show that the central conic $ Q(x,y) $ has its centre at the origin of the coordinate system if and only if
	\begin{align*} 
		(\mathcal{D} = 0) \wedge (\mathcal{E} = 0).
	\end{align*}
	The coordinates $(x_c,y_c)$ of the central conic's centre-point may be computed using the coefficients \eqref{Conic_xy} as follows
	\begin{align}
		x_c = \frac{\mathcal{B}\mathcal{E} - 2\mathcal{CD}}{4\mathcal{AC}-\mathcal{B}^2}, \label{xc} \\
		y_c = \frac{\mathcal{BD} - 2\mathcal{AE}}{4\mathcal{AC}-\mathcal{B}^2}. \label{yc}
	\end{align}
	Let us use note that the denominator $ 4\mathcal{AC}-\mathcal{B}^2 \neq 0 $ for all central conics. Using the above formulae we can prove the theorem in two steps:
	\\\\
	I. $ (\mathcal{D} = 0) \wedge (\mathcal{E} = 0) \implies Q(x,y) $ has the centre at the origin:
	\begin{quote}
		After direct substitution into formulae \eqref{xc}, \eqref{yc} we obtain immediately that
		\begin{align*}
			(x_c, y_c) = (0,0).
		\end{align*}
	\end{quote}
	\noindent
	II. $Q(x,y) $ has the centre at the origin $ \implies (\mathcal{D} = 0) \wedge (\mathcal{E} = 0): $
	\begin{quote}
		If a central conic $Q(x,y) $ has the centre at the origin, the formulae \eqref{xc},\eqref{yc} must be equal to zero, thus after multiplying both equations by their non-zero denominator, we get
		\begin{align}
			\mathcal{BE} - 2\mathcal{CD} = 0, \label{xc0} \\
			\mathcal{BD} - 2\mathcal{AE} = 0. \label{yc0}
		\end{align}
		To prove that the coefficients $ \mathcal{D} $ and $ \mathcal{E} $ must be zero for a central conic, we will inspect behaviour of $ \mathcal{D} $ and $ \mathcal{E} $ depending on whether the remaining coefficients $\mathcal{A,B,C}$ are or are not equal to zero, respectively. Therefore, 8 possibilities summarised by the following table may occur (`0' means equal to zero and `1' means non-zero):
		\begin{table}[H]
			\centering
			\setlength{\tabcolsep}{5pt}
			\begin{tabular}{@{}ccccccccc@{}}
				\toprule
				& 1. & 2. & 3. & 4. & 5. & 6. & 7. & 8. \\ \midrule
				$ \mathcal{A} $ & 0  & 1  & 0  & 0  & 1  & 0  & 1  & 1  \\
				$ \mathcal{B} $ & 0  & 0  & 1  & 0  & 1  & 1  & 0  & 1  \\
				$ \mathcal{C} $ & 0  & 0  & 0  & 1  & 0  & 1  & 1  & 1  \\ \bottomrule
			\end{tabular}
		\end{table} 
		\noindent
		Possibility 1. violates the definition of a conic itself, because at least one of the coefficients $\mathcal{ A,B,C} $ must be non-zero. Possibilities 2. and 4. cannot occur for central conics, since the denominator $ 4\mathcal{AC}-\mathcal{B}^2 $ must be non-zero. Possibility 3. yields
		\begin{align*}
			(\mathcal{BE} = 0) \wedge (\mathcal{B} \neq 0) \implies \mathcal{E}=0, \\
			(\mathcal{BD} = 0) \wedge (\mathcal{B} \neq 0) \implies \mathcal{D}=0, \\
		\end{align*}  
		so
		\begin{align*} 
			(\mathcal{D} = 0) \wedge (\mathcal{E} = 0).
		\end{align*}
		After inspecting 5., we get
		\begin{align*}
			(\mathcal{BE} = 0) \wedge (\mathcal{B} \neq 0) \implies \mathcal{E}=0, \\
			(\mathcal{BD}-2\mathcal{AE} = 0) \wedge (\mathcal{B} \neq 0) \wedge (\mathcal{E}=0) \implies \mathcal{D}=0,
		\end{align*}
		so
		\begin{align*} 
			(\mathcal{D} = 0) \wedge (\mathcal{E} = 0).
		\end{align*}
		The approach in 6. and 7. would be similar. Finally, in 8., we multiply the equation \eqref{xc0} by a non-zero coefficient $ \mathcal{B} $ and the equation \eqref{yc0} by a non-zero $ 2\mathcal{C} $ to get the system
		\begin{align*}
			\mathcal{B}^2\mathcal{E} - 2\mathcal{BCD} = 0, \\
			2\mathcal{BCD} - 4\mathcal{ACE} = 0.
		\end{align*}
		By summing up the resulting equations we get
		\begin{align*}
			\mathcal{B}^2\mathcal{E} - 4\mathcal{ACE} = 0,
		\end{align*}
		or 
		\begin{align*}
			\mathcal{E}(\mathcal{B}^2 - 4\mathcal{AC}) = 0.
		\end{align*}
		Since $ \mathcal{B}^2 - 4\mathcal{AC} \neq 0$ for central conics, it follows that $ \mathcal{E} = 0$. Because $\mathcal{E}=0$ and $\mathcal{B}\neq 0$, we conclude from the equation \eqref{yc0} \ that $ \mathcal{D} = 0 $ as well, so
		\begin{align*} 
			(\mathcal{D} = 0) \wedge (\mathcal{E} = 0).
		\end{align*} 
	\end{quote}
\end{proof}
Description of the origin-centred conic fitting will now proceed analogously to the problem with axes-aligned conic fitting.  After substituting the acquired condition into the vector \eqref{GACconic}, we obtain the IPNS representation of an origin-centred conic $ \Qo $ in the vector form
\begin{align}\label{IPNS_conic_centred}
	\Qo = 
	\setlength\arraycolsep{4pt}
	\begin{pmatrix}
		\vxb & \vmb & \vpb & 0 & 0 & \vp & 0 & 0
	\end{pmatrix}^T.
\end{align}
Consequently, for an origin-centred conic $ \Qo $ of the form \eqref{IPNS_conic_centred} and for given points represented by vectors $ P_i $ of the form \eqref{GACpoint}, we get the objective function
\begin{align}\label{cost_centred}
	\Qo\mapsto\sum_i(P_i\cdot \Qo)^2,
\end{align}
where $ \cdot $ represents the inner product between GAC vectors. Similarly to the previous problem, we minimise the objective function \eqref{cost_centred} subject to normalisation constraint
\begin{align} \label{constraint_centred}
	{\Qo}^2=1,
\end{align}
where the square in \eqref{constraint_centred} is w.r.t. the geometric product which is the same as the inner product in this case.

Because the objective function \eqref{cost_centred} in the matrix notation reads
\begin{align*}
	\Qo \mapsto {\Qo}^T P \Qo,
\end{align*}
the matrix $P$ in the form \eqref{Pentries} allows a partition according to the vector (34) of the origin-centred conic $ \Qo $. Since both $ \vi $ and $ \vii $ are zero, it follows that the rows 4 and 5 together with the columns 4 and 5 of the matrix $ P $ are of no importance to our problem. Therefore, we consider the following block decomposition:

\begin{align} \label{Pblocks_centred}
	P = 
	\renewcommand{\arraystretch}{1.2}
	\setlength\arraycolsep{3pt}
	\begin{pNiceArray}{cc|c|cc|c|cc}
		\Block{2-2}{P_0} &  & \sim & p_{41} & p_{51} & \sim & 0 & 0 \\
		 &  & \sim & p_{42} & p_{52} & \sim & 0 & 0 \\
		\hline
		\wr & \wr & \bullet & p_{43} & p_{53} & \bullet & 0 & 0 \\
		\hline
		p_{41} & p_{42} & p_{43} & p_{44} & p_{54} & p_{64} & 0 & 0 \\
		p_{51} & p_{52} & p_{53} & p_{54} & p_{55} & p_{65} & 0 & 0 \\
		\hline
		\wr & \wr & \bullet & p_{64} & p_{65} & \bullet & 0 & 0 \\
		 \hline
		0 & 0 & 0 & 0 & 0 & 0 & 0 & 0 \\
		0 & 0 & 0 & 0 & 0 & 0 & 0 & 0 
	\end{pNiceArray},
\end{align}
where $ P_0 $ is a $ 2 \times 2 $ matrix. The cells inscribed with $ \sim $ indicate the elements of $ 2 \times 2 $ matrix $ P_1^{0} $ consisting of the two columns from the matrix $ P $ that are mutually isolated, i.e.
\begin{align*}
	P_1^{0} =
	\renewcommand{\arraystretch}{1.2}
	\setlength\arraycolsep{3pt}
	\begin{pmatrix}
		p_{31} & p_{61} \\
		p_{32} & p_{62}
	\end{pmatrix}.
\end{align*}
Due to the symmetry of matrix $ P $, the cells with $ \wr $ stands for the matrix $ {P_1^{0}}^T $. Similarly to $ {P_1^{0}} $, the isolated elements designated with $ \bullet $ comprises $ 2 \times 2 $ matrix $ P_c^{0} $ of the form
\begin{align*}
	P_c^{0} =
	\renewcommand{\arraystretch}{1.2}
	\setlength\arraycolsep{3pt}
	\begin{pmatrix}
		p_{33} & p_{63} \\
		p_{63} & p_{66}
	\end{pmatrix}.
\end{align*}
Let us remark that the matrices $ {P_1^{0}} $, $ {P_1^{0}}^T $ and $ {P_c^{0}} $ are created from the matrices $ {P_1} $, $ {P_1}^T $ and $ {P_c},$ respectively, used in the partition \eqref{Pblocks} by omitting the corresponding rows and/or columns in the matrix $ P $. Submatrix $ P_0 $ remains the same as in the original decomposition.

Next, let us recall that the zeroness of variables $ \vi $ and $ \vii $ in the vector \eqref{IPNS_conic_centred} of conic $ \Qo $ also affects the final form of the normalisation constraint \eqref{constraint_centred}. By default, the matrix formulation of this constraint is
\begin{align*}
	{\Qo}^T B \Qo = 1,
\end{align*}
but thanks to sparseness of the matrix $ B $, it can be simplified. After defining a vector
\begin{align*} 
		\vo = 
		\setlength\arraycolsep{3pt}
		\begin{pmatrix}
			\vpb & \vp
		\end{pmatrix}^T
\end{align*}
and a $ 2 \times 2 $ matrix 
\begin{align}\label{Bco}
	\Bco = 
	\renewcommand{\arraystretch}{1.2}
	\setlength\arraycolsep{3pt}
	\begin{pmatrix}
		\phantom{-}0 & -1\phantom{|} \\
		-1  & \phantom{-}0\phantom{|}
	\end{pmatrix},
\end{align}
the normalisation constraint can be reformulated as 
\begin{equation} \label{hvezda_centred}
	{\vo}^T \Bco \vo = 1.
\end{equation}
Analogously to the matrices $ {P_1^{0}} $, $ {P_1^{0}}^T $ and $ {P_c^{0}} $, the vector $ \vo $ and the matrix $ \Bco $ were created from the vector $ v $ and matrix $ B_c $ by omitting those rows and/or columns corresponding to the zero variables $ \vi $ and $ \vii $.

\begin{prop} \label{solution_centred}
	The solution to the optimisation problem \eqref{cost_centred}, \eqref{constraint_centred} for conic fitting in GAC is given by
	$\Qo = \setlength\arraycolsep{3pt} \begin{pmatrix} w & \vpb & 0 & 0 & \vp & 0 & 0 \end{pmatrix}^T,$ where $ 
	\begin{pmatrix} \vpb & \vp \end{pmatrix}^T= \vo $ constitutes a 2-dimensional eigenvector corresponding to the minimal non-negative eigenvalue of 
	\begin{align} \label{Pcon_centred}
		P_{con}^{0} = \Bco ({P_c}^0-{P_1^{0}}^T{P_0}^{-1}P_1^{0})
	\end{align}
	and $ w = \begin{pmatrix} \vxb & \vmb \end{pmatrix}^T $ is a 2-dimensional vector computed as
	\begin{align} \label{w_centred}
		w = -P_0^{-1} P_1^{0} \vo.
	\end{align}
\end{prop}

\begin{proof}
	As in Proposition~\ref{solution}, the inversion in \eqref{w_centred} exists up to the case when all points lie on a single line.
	
	Again, using the method of Lagrange multipliers, we obtain the Lagrange function in the form
	$$
	P_\lambda(\Qo) := {\Qo}^TP\Qo + \lambda (1-{\Qo}^TB\Qo).
	$$
	and after applying the block decomposition \eqref{Pblocks_centred}, it can be rewritten as
	$$
	P_\lambda(\Qo) = w^T P_0 w + 2w^T P_1^{0} \vo + {\vo}^T P_c^0 \vo + \lambda(1-{\vo}^T \Bco \vo).
	$$
	We find its stationary points by differentiation w.r.t. the unknown vector $\Qo$ and solving the corresponding system of linear equations. Rest of the proof is analogous to proof of Proposition~\ref{solution}. 
\end{proof}

Not only that Proposition~\ref{solution_centred} formulates the solution to our fitting problem in a quite compact way, but it can also be used as a foundation for even more direct way of computation of the origin-centred conic $ \Qo. $ With the use of matrix $ P_{con}^{0} $ of the form \eqref{Pcon_centred} and its elements denoted as
\begin{align} \label{Pcon_centred_elements}
	P_{con}^{0} = 
	\renewcommand{\arraystretch}{1.1}
	\setlength\arraycolsep{3pt}
	\begin{pmatrix}
		a & r \\
		s & a
	\end{pmatrix},
\end{align}
we can indeed derive another method of computation summarised in the following proposition.
Moreover, although the proposition is based on the eigenvalue problem formulation, it reaches the solution without computing a single eigenvalue.

\begin{prop} \label{solution_centred_substitution}
	The solution to the optimisation problem \eqref{cost_centred}, \eqref{constraint_centred} for conic fitting in GAC is given by
	$\Qo =
	\setlength\arraycolsep{3pt}\begin{pmatrix} w & \vpb & 0 & 0 & \vp & 0 & 0 \end{pmatrix}^T,$ where $ 
	\setlength\arraycolsep{3pt}\begin{pmatrix} \vpb & \vp \end{pmatrix}^T = \vo $ is a vector obtained as 
	\begin{align} \label{vo_substitution}
		\vo = 
		\renewcommand{\arraystretch}{1.7}
		\begin{pmatrix}
			\vpb \\
			\vp
		\end{pmatrix}
		=
		\renewcommand{\arraystretch}{1.7}
		\begin{pmatrix}
			\pm \sqrt[4]{\tfrac{r}{4s}} \\
			\mp \tfrac{1}{2 \sqrt[4]{\tfrac{r}{4s}}}
		\end{pmatrix}
	\end{align}
	and $ w = \begin{pmatrix} \vxb & \vmb \end{pmatrix}^T $ is a vector of the form
	\begin{align} \label{w_centred_substitution}
		w = -P_0^{-1} P_1^{0} \vo.
	\end{align}
\end{prop}

\begin{proof}
	The eigenvalue problem derived in Proposition~\ref{solution_centred} reads
	\begin{equation*}
		P_{con}^{0} \vo = \lambda \vo,
	\end{equation*}
	or in the equivalent form
	\begin{equation*}
		(P_{con}^{0}- \lambda E) \vo = 0,
	\end{equation*}
	which can be expressed as a system of two equations:
	\begin{align}
		(a-\lambda) \vpb + r \vp = 0, \label{substitution_centred_eq1}\\
		s \vpb + (a-\lambda) \vp = 0. \label{substitution_centred_eq2}
	\end{align}
	After multiplying the equation \eqref{substitution_centred_eq1} by $ -\vp $ and the equation \eqref{substitution_centred_eq2} by $ \vpb $ we obtain 
	\begin{align*}
		-(a-\lambda) \vpb \vp - {r {\vp}}^2 = 0, \\
		{s {\vpb}}^2 + (a-\lambda) \vpb \vp = 0, 
	\end{align*}
	and summing the two equations up yields
	\begin{equation} \label{substitution_centred_eq3}
		{s {\vpb}}^2 - {r {\vp}}^2 = 0.
	\end{equation}
	Since $ \vo = \begin{pmatrix} \vpb & \vp \end{pmatrix}^T $ must comply with the normalisation constraint \eqref{hvezda_centred}, which can be rewritten simply as
	\begin{equation*} 
		-2 \vpb \vp = 1,
	\end{equation*} 
	either of the variables $ \vpb $ and $ \vp $ can be directly expressed in terms of the other one, for example
	\begin{equation} \label{vpb_substitution_centred}
		\vp = -\frac{1}{2\vpb}.
	\end{equation}  
	Substitution of \eqref{vpb_substitution_centred} into \eqref{substitution_centred_eq3} gives 
	\begin{equation*} 
		{s {\vpb}}^2 - \frac{r}{{4 {\vpb}}^2} = 0,
	\end{equation*}
	and after multiplying by $ {4 {\vpb}}^2 $ we have
	\begin{equation*} 
		4 {s {\vpb}}^4 - r = 0,
	\end{equation*}
	implying that
	\begin{equation*} 
		\vpb = \pm \sqrt[4]{\frac{r}{4s}},
	\end{equation*}
	which together with \eqref{vpb_substitution_centred} yields
	\begin{equation*}
		\vp = \mp \frac{1}{2 \sqrt[4]{\tfrac{r}{4s}}}.
	\end{equation*}
	The computation of the vector $ w $ given by \eqref{w_centred_substitution} is the same as in the Proposition~\ref{solution_centred}, where it was derived and proven valid.
\end{proof}

Let us remark that the notation of entries of matrix $ P_{con}^{0} $ in \eqref{Pcon_centred_elements} is intentional---since we denote the entries of matrix $ P_{con} $ of the form \eqref{Pcon} in the following way:
\begin{align} \label{Pcon_elements}
	P_{con} = 
	\renewcommand{\arraystretch}{1.2}
	\setlength\arraycolsep{2pt}
	\begin{pmatrix}
		a & \phantom{-}k & \phantom{-}l & \phantom{-}r \\
		m & \phantom{-}t & \phantom{-}b & -k \\
		n & \phantom{-}b & \phantom{-}u & -l \\
		s & -m & -n & \phantom{-}a
	\end{pmatrix},
\end{align}
it is obvious that the matrix $ P_{con}^{0} $ consists of four corner elements of the matrix $ P_{con} $. This is a consequence of the block decomposition \eqref{Pblocks_centred} of matrix $ P $ where the two rows and columns corresponding to the zero variables $ \vi $ and $ \vii $ are omitted.

Apart from two equivalent solutions formulated in Propositions~\ref{solution_centred} and \ref{solution_centred_substitution} we can reach the same solution by yet another way. In contrast with mentioned Propositions, in which we just used subparts of matrix $ P $ to reformulate the original fitting problem, we can use the original fitting approach and extend the point dataset so that we necessarily obtain an origin-centred conic as a result.

\begin{prop} \label{solution_centred_symmetrisation}
	The solution to the optimisation problem \eqref{cost_centred}, \eqref{constraint_centred} for conic fitting in GAC is given by Proposition~\ref{solution} provided that the dataset consisting of points $ p_i = (x_i,y_i) $, $ i=1,\ldots, N $, is extended by points $ p_i^{--} = (-x_i,-y_i) $, $ i=1,\ldots, N $, which constitute centrally symmetric images of points $ p_i $ w.r.t. the origin of the coordinate system.
\end{prop}
\begin{proof}
	Firstly, let us recall that every point $ p_i = (x_i,y_i) $ is in GAC represented by the vector
	\begin{align*}
		P_i=&
		\setlength\arraycolsep{4pt}
		\begin{pmatrix}
			0 &0& 1& x_i& y_i& \frac12(x_i^2+y_i^2)& \frac12(x_i^2-y_i^2)& x_iy_i
		\end{pmatrix}^T
	\end{align*}
	therefore, every point $ p_i^{--} = (-x_i,-y_i) $ is represented by a vector $ P_i^{--} $ of the form
	\begin{align*}
		P_i^{--}=&
		\setlength\arraycolsep{4pt}
		\begin{pmatrix}
			0 &0& 1& -x_i& -y_i& \frac12(x_i^2+y_i^2)& \frac12(x_i^2-y_i^2)& x_iy_i
		\end{pmatrix}^T.
	\end{align*} 
	Consequently, the corresponding matrix $ P_{con} $ is of the form
	\begin{align*} 
		P_{con}^{0,sym} = 
		\renewcommand{\arraystretch}{1.2}
		\begin{pmatrix}
			a^{*} & 0\phantom{^{*}} & 0\phantom{^{*}} & r^{*} \\
			0\phantom{^{*}} & t^{*} & b^{*} & 0\phantom{^{*}} \\
			0\phantom{^{*}} & b^{*} & u^{*} & 0\phantom{^{*}} \\
			s^{*} & 0\phantom{^{*}} & 0\phantom{^{*}} & a^{*}
		\end{pmatrix},
	\end{align*}
	where superscript asterisks are used to stress that the elements of matrix $ P_{con}^{0,sym} $ are generally different than those of the matrix $ P_{con} $ of the form \eqref{Pcon_elements} constructed from the points $ p_i $ only. Fortunately, after further inspection, one can conclude that some elements of the matrix $ P_{con}^{0,sym} $ can be expressed in terms of the entries of matrix $ P_{con} $. In particular, we get 
	\begin{align*}
		P_{con}^{0,sym} = 
		\renewcommand{\arraystretch}{1.2}
		\begin{pmatrix}
			a^{*} & 0\phantom{^{*}} & 0\phantom{^{*}} & r^{*} \\
			0\phantom{^{*}} & t^{*} & b^{*} & 0\phantom{^{*}} \\
			0\phantom{^{*}} & b^{*} & u^{*} & 0\phantom{^{*}} \\
			s^{*} & 0\phantom{^{*}} & 0\phantom{^{*}} & a^{*}
		\end{pmatrix} 
		=
		\renewcommand{\arraystretch}{1.2}
		\begin{pmatrix}
			2a & 0\phantom{^{*}} & 0\phantom{^{*}} & 2r \\
			0 & t^{*} & b^{*} & 0 \\
			0 & b^{*} & u^{*} & 0 \\
			2s & 0\phantom{^{*}} & 0\phantom{^{*}} & 2a
		\end{pmatrix}.
	\end{align*}
	Due to the special form of the matrix $ P_{con}^{0,sym} $, the set $ V^{*} $ of its eigenvectors 
	\begin{align*}
		v_1^{*} = 
		\renewcommand{\arraystretch}{1.2}
		\begin{pmatrix}
			v_{11}^{*} \\
			v_{12}^{*} \\
			v_{13}^{*} \\
			v_{14}^{*}
		\end{pmatrix}, 
		\ldots, 
		v_4^{*} = 
		\renewcommand{\arraystretch}{1.2}
		\begin{pmatrix}
			v_{41}^{*} \\
			v_{42}^{*} \\
			v_{43}^{*} \\
			v_{44}^{*} 
		\end{pmatrix}, 		
	\end{align*}
	can be described as
	\begin{align*}
		V^{*} = 
		\begin{pmatrix}
			v_1^{*} & v_2^{*} & v_3^{*} & v_4^{*}
		\end{pmatrix} 
		=
		\renewcommand{\arraystretch}{1.2}
		\begin{pmatrix}
			v_{11}^{*} & 0 & 0 & v_{41}^{*} \\
			0 & v_{22}^{*} & v_{32}^{*} & 0 \\
			0 & v_{23}^{*} & v_{33}^{*} & 0 \\
			v_{14}^{*} & 0 & 0 & v_{44}^{*}
		\end{pmatrix}
	\end{align*}
	with $ \lambda_1^{*}, \ldots, \lambda_4^{*} $ as their corresponding eigenvalues. From the geometric nature of the problem, the dataset centrally symmetric w.r.t. the coordinate origin must necessarily be fitted with origin-centred conic, therefore, neither of the eigenvectors $ v_2^{*} $ and $ v_3^{*} $ can be the solution, since they do not correspond to an origin-centred conic (also, neither of the eigenvalues $ \lambda_2^{*} $ and $ \lambda_3^{*} $ can represent the optimal solution). Regarding the eigenvalues $ \lambda_1^{*} $ and $ \lambda_4^{*} $, they are computed as
	\begin{align*}
		\renewcommand{\arraystretch}{1.4}
		\begin{pmatrix}
			\lambda_1^{*} \\
			\lambda_4^{*}
		\end{pmatrix} 
		=
		\begin{pmatrix}
			2(a + \sqrt{rs}) \\
			2(a - \sqrt{rs}) 
		\end{pmatrix}
	\end{align*} 
	and their corresponding eigenvectors as
	\begin{align*}
		\setlength\arraycolsep{4pt}
		\begin{pmatrix}
			v_1^{*} & v_4^{*}
		\end{pmatrix} 
		=
		\renewcommand{\arraystretch}{1.2}
		\begin{pmatrix}
			\sqrt{\tfrac{r}{s}} & -\sqrt{\tfrac{r}{s}} \\
			0 & 0 \\
			0 & 0 \\
			1 & 1		
		\end{pmatrix},
	\end{align*} 
	while the eigenvalues $ \lambda_1^{0} $ and $ \lambda_2^{0} $ of the matrix $ P_{con}^{0} $ can be expressed as 
	\begin{align*}
		\renewcommand{\arraystretch}{1.4}
		\begin{pmatrix}
			\lambda_1^{0} \\
			\lambda_2^{0}
		\end{pmatrix} 
		=
		\begin{pmatrix}
			a + \sqrt{rs} \\
			a - \sqrt{rs}
		\end{pmatrix}
	\end{align*}
 	with the corresponding eigenvectors $ v_1^{0} $ and $ v_2^{0} $ of the form
 	\begin{align*}
 		\setlength\arraycolsep{4pt}
 		\begin{pmatrix}
 			v_1^{0} & v_2^{0}
 		\end{pmatrix} 
 		=
 		\renewcommand{\arraystretch}{1.2}
 		\begin{pmatrix}
 			\sqrt{\tfrac{r}{s}} & -\sqrt{\tfrac{r}{s}} \\
 			1 & 1		
 		\end{pmatrix}.
 	\end{align*} 
 	We can see that 
 	\begin{align*}
 		\renewcommand{\arraystretch}{1.4}
 		\begin{pmatrix}
 			\lambda_1^{*} \\
 			\lambda_4^{*}
 		\end{pmatrix} 
 		= 2
 		\begin{pmatrix}
 			\lambda_1^{0} \\
 			\lambda_2^{0}
 		\end{pmatrix} 
 	\end{align*}
 	and that  
 	\begin{align*}
 		\setlength\arraycolsep{4pt}
 		\begin{pmatrix}
 			v_1^{*} & v_4^{*}
 		\end{pmatrix} 
 		\approx
 		\begin{pmatrix}
 			v_1^{0} & v_2^{0}
 		\end{pmatrix},
 	\end{align*} 
 	where by $ \approx $ we mean that after omitting zero elements in the eigenvectors $ v_1^{*} $ and $ v_4^{*} $, the eigenvectors of both matrices are the same (note that in the case of eigenvectors $ v_1^{0} $ and $ v_2^{0}, $ two zeroes omitted in the eigenvectors $ v_1^{*} $ and $ v_4^{*} $ would eventually be added anyway, in order to form a complete conic vector). Consequently, the optimal eigenvalue of the matrix $ P_{con}^{0,sym} $ corresponds to the same solution as the one obtained by using the matrix $ P_{con}^{0} $ in Proposition~\ref{solution_centred}, which is also equivalent to the solution acquired in Proposition~\ref{solution_centred_substitution}. 
\end{proof}
Despite the tempting simplicity of this approach, we must keep in mind that the solution in Proposition~\ref{solution_centred_symmetrisation} is reached by doubling the number of points in  the dataset. Such an enlargement of the input data necessarily results into the increase of computational demands, which is generally preferred to be avoided, especially when having a large number of points at the beginning already.

\subsection*{Implementation}
Since we have formulated three ways of fitting an origin-centred conic $ \Qo $, we describe three corresponding algorithms:
\begin{itemize}
	\item[$\bullet$] \textbf{Q0} --- via eigenvalues and eigenvectors from Proposition~\ref{solution_centred} \\
	\item[$\bullet$] \textbf{Q0-dir} --- `direct' computation using Proposition~\ref{solution_centred_substitution} without computing eigenvalues \\
	\item[$\bullet$] \textbf{Q0-sym} --- central symmetrisation of the dataset according to Proposition~\ref{solution_centred_symmetrisation} 
\end{itemize}
As in the case of \algorithmQAL, we use in Algorithms \textbf{Q0} and \textbf{Q0-dir} the reduced forms of vectors and matrices. Therefore, we define the $ i $-th reduced GAC point
\begin{equation*} 
	\hat{P}_i^{0} =
	\setlength\arraycolsep{4pt} 
	\begin{pmatrix}
		1& \frac12(x_i^2+y_i^2)& \frac12(x_i^2-y_i^2)& x_iy_i
	\end{pmatrix}^T
\end{equation*}
and create the corresponding reduced data matrix $ \hat{D}^{0} $ of the size $ 4 \times N $, where the $ i $-th column is $ \hat{P}_i^{0} $. Additionally, by removing the rows 4, 5,7 and 8 and columns 1,2, 4 and 5 from the matrix $ B $ of the form \eqref{B}, we also define a reduced matrix for GAC inner product
\begin{align*}
	\hat{B}^{0} = 
	\renewcommand{\arraystretch}{1.2}
	\setlength\arraycolsep{3pt}
	\begin{pmatrix}
		\phantom{-}0&\phantom{-}0&\phantom{-}0&-1\phantom{|}\\
		\phantom{-}0&\phantom{-}0&-1&\phantom{-}0\phantom{|}\\
		\phantom{-}0&-1&\phantom{-} 0&\phantom{-}0\phantom{|}\\
		-1&\phantom{-}0&\phantom{-} 0&\phantom{-}0\phantom{|} 
	\end{pmatrix}.
\end{align*} 
A reduction of the matrix $ B_c $ used in the following two algorithms is the matrix $ \Bco $ of the form \eqref{Bco}.
\\
\rule{\linewidth}{1.5pt}
\\
\textbf{\algorithmQo}
\\
\rule[6pt]{\linewidth}{0.75pt} 
\\
\begin{minipage}{\textwidth}
	\begin{minipage}{.05\textwidth}
		\lstset{
			showlines=true
		}
		\begin{lstlisting}[style=Matlab-editor, rulecolor=\color{white}, mathescape]
$1.$




		\end{lstlisting}
	\end{minipage}
	\hfill
	\begin{minipage}{.95\textwidth}
		\begin{lstlisting}[framexrightmargin=-5.5pt, style=Matlab-editor]
B = zeros(4);
I2 = [0 1; 1 0];
B(1:2,3:4) = -I2;
B(3:4,1:2) = -I2;
Bc = -I2;
		\end{lstlisting}
	\end{minipage}
\end{minipage}
\begin{minipage}{\textwidth}
	\begin{minipage}{.05\textwidth}
		\lstset{
			showlines=true
		}
		\begin{lstlisting}[style=Matlab-editor, rulecolor=\color{white}, mathescape]
$2.$
			
			
			
		\end{lstlisting}
	\end{minipage}
	\hfill
	\begin{minipage}{.95\textwidth}
		\begin{lstlisting}[framexrightmargin=-5.5pt, style=Matlab-editor]
D = ones(4,N);
D(2,:) = 1/2*(p(1,:).^2 + p(2,:).^2);
D(3,:) = 1/2*(p(1,:).^2 - p(2,:).^2);
D(4,:) = p(1,:).*p(2,:);
		\end{lstlisting}
	\end{minipage}
\end{minipage}
\begin{minipage}{\textwidth}
	\begin{minipage}{.05\textwidth}
		\lstset{
			showlines=true
		}
		\begin{lstlisting}[style=Matlab-editor, rulecolor=\color{white}, mathescape]
$3.$
			
			
			
		\end{lstlisting}
	\end{minipage}
	\hfill
	\begin{minipage}{.95\textwidth}
		\begin{lstlisting}[framexrightmargin=-5.5pt, style=Matlab-editor]
P = 1/N*B*(D*D')*B';
Pc = P(3:4,3:4);
P0 = P(1:2,1:2);
P1 = P(1:2,3:4);
		\end{lstlisting}
	\end{minipage}
\end{minipage}
\begin{minipage}{\textwidth}
	\begin{minipage}{.05\textwidth}
		\lstset{
			showlines=true
		}
		\begin{lstlisting}[style=Matlab-editor, rulecolor=\color{white}, mathescape]
$4.$
			
			
		\end{lstlisting}
	\end{minipage}
	\hfill
	\begin{minipage}{.95\textwidth}
		\begin{lstlisting}[framexrightmargin=-5.5pt, style=Matlab-editor]
Pcon = Bc*(Pc-P1'*(P0\P1));
[EV,ED] = eig(Pcon);
EW = diag(ED);	
		\end{lstlisting}
	\end{minipage}
\end{minipage}
\begin{minipage}{\textwidth}
	\begin{minipage}{.05\textwidth}
		\lstset{
			showlines=true
		}
		\begin{lstlisting}[style=Matlab-editor, rulecolor=\color{white}, mathescape]
$5.$
			
		\end{lstlisting}
	\end{minipage}
	\hfill
	\begin{minipage}{.95\textwidth}
		\begin{lstlisting}[framexrightmargin=-5.5pt, style=Matlab-editor]
k_opt = find(EW == min(EW(EW>0)));
v_opt = EV(:,k_opt);		
		\end{lstlisting}
	\end{minipage}
\end{minipage}
\begin{minipage}{\textwidth}
	\begin{minipage}{.05\textwidth}
		\lstset{
			showlines=true
		}
		\begin{lstlisting}[style=Matlab-editor, rulecolor=\color{white}, mathescape]
$6.$
			
		\end{lstlisting}
	\end{minipage}
	\hfill
	\begin{minipage}{.95\textwidth}
		\begin{lstlisting}[framexrightmargin=-5.5pt, style=Matlab-editor]
kappa = v_opt'*Bc*v_opt;
v_opt = 1/sqrt(kappa)*v_opt; 
		\end{lstlisting}
	\end{minipage}
\end{minipage}
\begin{minipage}{\textwidth}
	\begin{minipage}{.05\textwidth}
		\lstset{
			showlines=true
		}
		\begin{lstlisting}[style=Matlab-editor, rulecolor=\color{white}, mathescape]
$7.$
			
		\end{lstlisting}
	\end{minipage}
	\hfill
	\begin{minipage}{.95\textwidth}
		\begin{lstlisting}[framexrightmargin=-5.5pt, style=Matlab-editor]
w = -P0\P1*v_opt;
Q = [w;v_opt(1);0;0;v_opt(2);0;0];
		\end{lstlisting}
	\end{minipage}
\end{minipage}
\\
\rule[12pt]{\linewidth}{1.5pt}
Computation of the matrix $ P $ in step 3. then yields a $ 4 \times 4$ reduced matrix of the form
\begin{align*} 
	\hat{P}^{0}=
	\renewcommand{\arraystretch}{1.6}
	\setlength\arraycolsep{2pt}
	\begin{pmatrix}
		P_0 & P_1^{0}\\
		{P_1^{0}}^T & P_c^{0}\\
	\end{pmatrix}.
\end{align*}
Consequently, after defining the reduced GAC vector of the origin-centred conic
\begin{align*}
	\hat{Q}^{0}=&
	\setlength\arraycolsep{4pt}
	\begin{pmatrix}
		\vxb & \vmb & \vpb & \vp 
	\end{pmatrix}^T,
\end{align*}
the objective function \eqref{cost_centred} can be equivalently computed as
\begin{align*}
	\hat{Q}^{0} \mapsto \hat{Q}^{{0}^T} \hat{P}^{0} \hat{Q}^{0}.
\end{align*}
Due to Proposition~\ref{solution_centred_substitution} we can reuse the first steps of \algorithmQo{} in \algorithmQodir, while the other steps of the algorithm will be altered.
\\
\rule{\linewidth}{1.5pt}
\\
\textbf{\algorithmQodir}
\\
\rule[6pt]{\linewidth}{0.75pt} \vspace*{-21pt}
\begin{enumerate}[leftmargin=1.05cm]
	\item[1.–3.] See \algorithmQo.	
	\vspace{-0.6em}
	\item[4.] Formation of matrix $P_{con}^0$ according to \eqref{Pcon_centred} and extraction of its elements $ r $ and $ s $.
	\begin{lstlisting}[framexrightmargin=-3pt, style=Matlab-editor]
Pcon = Bc*(Pc-P1'*(P0\P1));
r = Pcon(1,2);
s = Pcon(2,1);
	\end{lstlisting}
	\vspace{-0.6em}	
	\item[5.] Direct computation of the optimal vector  $ v^{0} $  given by \eqref{vo_substitution}.
	\begin{lstlisting}[framexrightmargin=-3pt, style=Matlab-editor]
vpb = (r/(4*s))^0.25;  
vp = -1/(2*vpb);                      
v_opt = [vpb;vp];
	\end{lstlisting}
	\vspace{-0.6em}		
	\item[6.] Computation of $w$ by \eqref{w_centred_substitution} and forming the optimal vector $ \Qo $ according to Proposition~\ref{solution_centred_substitution}
	\begin{lstlisting}[framexrightmargin=-3pt, style=Matlab-editor]
w = -P0\P1*v_opt;
Q = [w;vpb;0;0;vp;0;0];
	\end{lstlisting}		
\end{enumerate}
\rule[12pt]{\linewidth}{1.5pt} 
Finally, we present a MATLAB code for \algorithmQosym, which is basically \algorithmQ{} preceded by prior extension of the dataset. Let us note that the reduced forms of vectors and matrices used in the following algorithm correspond to the reduced forms described in the commentary to \algorithmQ.
\\
\rule{\linewidth}{1.5pt}
\\
\textbf{\algorithmQosym}
\\
\rule[6pt]{\linewidth}{0.75pt} \vspace*{-21pt}
\begin{enumerate}[leftmargin=1.05cm]
	\item[0.] Extension of the point dataset by its centrally symmetric images w.r.t. the origin of the coordinate system.
	\begin{lstlisting}[framexrightmargin=-3pt, style=Matlab-editor]
p = [p,-p];
N = 2*N;
	\end{lstlisting}
	\vspace{-0.6em}
	\item[1.–2.] See steps 1* and 2* from the commentary to \algorithmQ 
	\vspace{-0.6em}
	\item[3.–7.] See \algorithmQ
\end{enumerate}
\rule[12pt]{\linewidth}{1.5pt} 

\subsection{Axes-aligned origin-centred conic}

\begin{theorem} 
	A central conic $ Q $ represented in the form of a vector \eqref{GACconic} has a centre at the origin of coordinate system and is axes-aligned if and only if 
	\begin{align*}
		(\vxb = 0) \wedge (\vi = 0) \wedge (\vii = 0).
	\end{align*}
\end{theorem}
\begin{proof}
	This Theorem is a direct consequence of Theorems~\ref{theorem_aligned} and \ref{theorem_centred}. 
\end{proof}

It is evident that to describe the fitting of an axes-aligned origin-centred conic it is efficient to merge the approaches from two previous methods. After substituting the geometric conditions into the vector \eqref{GACconic}, we obtain the IPNS representation of an axes-aligned origin-centred conic $ \Qalo $
\begin{align}\label{IPNS_conic_aligned_centred}
	\Qalo = 
	\setlength\arraycolsep{3pt}
	\begin{pmatrix}
		0 & \vmb & \vpb & 0 & 0 & \vp & 0 & 0
	\end{pmatrix}^T.
\end{align}
Hence, for an axes-aligned origin-centred conic $ \Qalo $ of the form \eqref{IPNS_conic_aligned_centred} and for given points represented by vectors $ P_i $ of  the form \eqref{GACpoint}, we get the objective function
\begin{align}\label{cost_aligned_centred}
	\Qalo\mapsto\sum_i(P_i\cdot \Qalo)^2
\end{align}
with $ \cdot $ representing the inner product between GAC vectors. The normalisation constraint reads
\begin{align} \label{constraint_aligned_centred}
	{\Qalo}^2=1,
\end{align}
where the square in \eqref{constraint_aligned_centred} is w.r.t. the geometric product which is the same as the inner product in this case.

Since the elements $ \vxb $, $ \vi $ and $ \vii $ in the vector $ \Qalo $ of an axes-aligned origin-centred conic are all zero, in the partition of the matrix $ P $ we omit the rows and columns 1, 4 and 5. Therefore, the resulting block decomposition of the matrix $ P $ is of the form 

\begin{align} \label{Pblocks_aligned_centred}
	P = 
	\renewcommand{\arraystretch}{1.3}
	\setlength\arraycolsep{3pt}
	\begin{pNiceArray}{c|c|c|cc|c|cc}
		p_{11} & p_{21} & p_{31} & p_{41} & p_{51} & p_{61} & 0 & 0 \\
		\hline
		p_{21} & \Block{1-1}{P_0^{al}} & \sim & p_{42} & p_{52} & \sim & 0 & 0 \\
		\hline
		p_{34} & \wr & \bullet & p_{43} & p_{53} & \bullet & 0 & 0 \\
		\hline
		p_{41} & p_{42} & p_{43} & p_{44} & p_{54} & p_{64} & 0 & 0 \\
		p_{51} & p_{52} & p_{53} & p_{54} & p_{55} & p_{65} & 0 & 0 \\
		\hline
		p_{61} & \wr & \bullet & p_{64} & p_{65} & \bullet & 0 & 0 \\
		\hline
		0 & 0 & 0 & 0 & 0 & 0 & 0 & 0 \\
		0 & 0 & 0 & 0 & 0 & 0 & 0 & 0 
	\end{pNiceArray},
\end{align}
where $ P_0^{al} $ is the same real number as in the decomposition \eqref{Pblocks_aligned}, the cells denoted $ \bullet $ constitutes $ 2 \times 2 $ matrix $ P_c^0 $ from decomposition \eqref{Pblocks_centred} and the cells inscribed with $ \sim $ indicate the elements of vector $ P_1^{al,0} $, i.e.
\begin{align*}
	P_1^{al,0} =
	\setlength\arraycolsep{3pt}
	\begin{pmatrix}
		p_{32} & p_{62}
	\end{pmatrix}.
\end{align*}
Symmetrically, cells with $ \wr $ stand for the matrix $ {P_1^{al,0}}^T $.

The default matrix form of the normalisation constraint \eqref{constraint_aligned_centred} is 
\begin{align*}
	{Q^{al,o}}^T B Q^{al,o} = 1
\end{align*}
and due to the zero variables in  vector \eqref{IPNS_conic_aligned_centred} of the conic it can be rewritten as
\begin{equation} \label{hvezda_aligned_centred_simplified} 
	{\vo}^T \Bco \vo = 1,
\end{equation}
which coincides with the normalisation constraint \eqref{constraint_centred} used in the origin-centred conic fitting problem.

\begin{prop} \label{solution_aligned_centred}
	The solution to the optimisation problem \eqref{cost_aligned_centred}, \eqref{constraint_aligned_centred} for conic fitting in GAC is given by
	$Q^{al,0} = \setlength\arraycolsep{3pt}\begin{pmatrix} 0 & \wal & \vpb & 0 & 0 & \vp & 0 & 0 \end{pmatrix}^T,$ where $ 
	\setlength\arraycolsep{3pt}\begin{pmatrix} \vpb & \vp \end{pmatrix}^T = \vo $ constitutes a 2-dimensional eigenvector corresponding to the minimal non-negative eigenvalue of 
	\begin{align} \label{Pcon_aligned_centred}
		P_{con}^{al,0} = \Bco ({P_c}^0-{P_1^{al,0}}^T{P_0^{al}}^{-1}P_1^{al,0})
	\end{align}
	and $ \wal \equiv \vmb $ is a real number computed as
	\begin{align} \label{w_aligned_centred}
		\wal = -{P_0^{al}}^{-1} P_1^{al,0} \vo.
	\end{align}
\end{prop}

\begin{proof}
	Let us recall that, as well as in Proposition~\ref{solution_aligned}, the inversion in \eqref{w_aligned_centred} exists up to the case when all points lie on double-line $ x^2 - y^2 = 0 $.
	
	Using the method of Lagrange multipliers, we acquire the Lagrange function in the form
	$$
	P_\lambda(\Qalo) := {\Qalo}^TP\Qalo + \lambda (1-{\Qalo}^TB\Qalo)
	$$
	and after applying the block decomposition \eqref{Pblocks_aligned_centred}, it can be rewritten as
	$$
	P_\lambda(\Qalo) = {\wal}^T P_0^{al} \wal + 2{\wal}^T P_1^{al,0} \vo + {\vo}^T P_c^0 \vo + \lambda(1-{\vo}^T \Bco \vo).
	$$
	After differentiation w.r.t. the unknown vector $\Qalo$ and solving the resulting system of linear equations, the rest of the proof proceeds similarly to proof of Proposition~\ref{solution}. 
\end{proof}

As in the case of the origin-centred conic fitting, we apply simple normalisation constraint \eqref{hvezda_aligned_centred_simplified} and with the use of the matrix $ P_{con}^{al,0} $ in the form \eqref{Pcon_aligned_centred} reach the solution without computing a single eigenvalue.
Since the matrix $ P_{con}^{al,0} $ is generally different from $ P_{con}^{0} $ and does not consist of the entries of matrix $ P_{con} $, we must distinguish the notation of its elements as well, for instance by adding apostrophes:
\begin{align} \label{Pcon_aligned_centred_elements}
	P_{con}^{al,0} = 
	\renewcommand{\arraystretch}{1.2}
	\setlength\arraycolsep{3pt}
	\begin{pmatrix} 
		a' & r' \\
		s' & a'
	\end{pmatrix}.
\end{align}

\begin{prop} \label{solution_aligned_centred_substitution}
	The solution to the optimisation problem \eqref{cost_aligned_centred}, \eqref{constraint_aligned_centred} for conic fitting in GAC is given by
	$\Qalo = \setlength\arraycolsep{3pt} \begin{pmatrix} 0 & \wal & \vpb & 0 & 0 & \vp & 0 & 0 \end{pmatrix}^T,$ where $ 
	\setlength\arraycolsep{3pt} \begin{pmatrix} \vpb & \vp \end{pmatrix}^T = \vo $ is a vector obtained as 
	\begin{align} \label{valo_substitution}
		\vo = 
		\renewcommand{\arraystretch}{1.7}
		\begin{pmatrix}
			\vpb \\
			\vp
		\end{pmatrix}
		=
		\renewcommand{\arraystretch}{1.7}
		\begin{pmatrix}
			\pm \sqrt[4]{\tfrac{r'}{4s'}} \\
			\mp \tfrac{1}{2 \sqrt[4]{\tfrac{r'}{4s'}}}
		\end{pmatrix}
	\end{align}
	and $ \wal \equiv \vmb $ is a real number computed as
	\begin{align} \label{w_aligned_centred_substitution}
		\wal = -{P_0^{al}}^{-1} P_1^{al,0} \vo.
	\end{align}
\end{prop}

\begin{proof}
	The eigenvalue problem derived in Proposition~\ref{solution_aligned_centred} reads
	\begin{equation*}
		P_{con}^{al,0} \vo = \lambda \vo
	\end{equation*}
	and is almost the same as in the proof of Proposition~\ref{solution_centred_substitution}, only the matrix $ P_{con}^{0} $ is replaced with $ P_{con}^{al,0} $. Therefore, the rest of the proof for the vector $v^0$ computation proceeds in a way analogous to the Proposition~\ref{solution_centred_substitution}, only using the elements of $ P_{con}^{al,0} $. 
	
	Additionally, computation of the vector $ \wal $ given by \eqref{w_aligned_centred_substitution} is the same as in Proposition~\ref{solution_aligned_centred}.
\end{proof}

Similarly to fitting an origin-centred conic in subsection~\ref{subsec:4.2}, it is possible to obtain the solution using \algorithmQ{} on an appropriately extended dataset.
\begin{prop} \label{solution_aligned_centred_symmetrisation}
	The solution to the optimisation problem \eqref{cost_aligned_centred}, \eqref{constraint_aligned_centred} for conic fitting in GAC is given by Proposition~\ref{solution} provided that the dataset consisting of the points $ p_i = (x_i,y_i) $, $ i=1,\ldots, N $, is extended by points $ p_i^{--} = (-x_i,-y_i) $, $ i=1,\ldots, N $, which constitute centrally symmetric images of points $ p_i $ w.r.t. the origin of the coordinate system and by points $ p_i^{+-} = (x_i,-y_i) $, $ i=1,\ldots, N $, and points $ p_i^{-+} = (-x_i,y_i) $, $ i=1,\ldots, N $, which are axially symmetric images of points $ p_i $ w.r.t. y-axis and x-axis, respectively.
\end{prop}
\begin{proof}
	Proof is conducted in a way similar to the proof of Proposition~\ref{solution_centred_symmetrisation}. Let us note that the points $ p_i $, $ p_i{--} $, $ p_i{+-} $, $ p_i{-+} $ are in GAC represented by respective vectors
	\begin{align*}
		P_i=&
		\setlength\arraycolsep{4pt}
		\begin{pmatrix}
			0 &0& 1& \phantom{-}x_i& \phantom{-}y_i& \frac12(x_i^2+y_i^2)& \frac12(x_i^2-y_i^2)& x_iy_i
		\end{pmatrix}^T, \\
		P_i^{--}=&
		\setlength\arraycolsep{4pt}
		\begin{pmatrix}
			0 &0& 1& -x_i& -y_i& \frac12(x_i^2+y_i^2)& \frac12(x_i^2-y_i^2)& x_iy_i
		\end{pmatrix}^T, \\
		P_i^{+-}=&
		\setlength\arraycolsep{4pt}
		\begin{pmatrix}
			0 &0& 1& \phantom{-}x_i& -y_i& \frac12(x_i^2+y_i^2)& \frac12(x_i^2-y_i^2)& x_iy_i
		\end{pmatrix}^T, \\
		P_i^{-+}=&
		\setlength\arraycolsep{4pt}
		\begin{pmatrix}
			0 &0& 1& -x_i& \phantom{-}y_i& \frac12(x_i^2+y_i^2)& \frac12(x_i^2-y_i^2)& x_iy_i
		\end{pmatrix}^T.
	\end{align*}
	Consequently, the matrix $ P_{con} $ in our case is of the form
	\begin{align*} 
		P_{con}^{al,0,sym} = 
		\renewcommand{\arraystretch}{1.2}
		\begin{pmatrix}
			a^{**} & 0\phantom{^{**}} & 0\phantom{^{**}} & r^{**} \\
			0\phantom{^{**}} & t^{**} & b^{**} & 0\phantom{^{**}} \\
			0\phantom{^{**}} & b^{**} & u^{**} & 0\phantom{^{**}} \\
			s^{**} & 0\phantom{^{**}} & 0\phantom{^{**}} & a^{**}
		\end{pmatrix}.
	\end{align*}
	Fortunately, the corner elements of $ P_{con}^{al,0,sym} $ can be expressed in terms of the entries of $ P_{con}^{al,0} $ of the form \eqref{Pcon_aligned_centred_elements}. Specifically, we get
	\begin{align*}
		P_{con}^{al,0,sym} = 
		\renewcommand{\arraystretch}{1.2}
		\begin{pmatrix}
			a^{**} & 0\phantom{^{**}} & 0\phantom{^{**}} & r^{**} \\
			0\phantom{^{**}} & t^{**} & b^{**} & 0\phantom{^{**}} \\
			0\phantom{^{**}} & b^{**} & u^{**} & 0\phantom{^{**}} \\
			s^{**} & 0\phantom{^{**}} & 0\phantom{^{**}} & a^{**}
		\end{pmatrix}
		=
		\renewcommand{\arraystretch}{1.2}
		\begin{pmatrix}
			4a' & 0\phantom{^{*}} & 0\phantom{^{*}} & 4r' \\
			0 & t^{**} & b^{**} & 0 \\
			0 & b^{**} & u^{**} & 0 \\
			4s' & 0\phantom{^{*}} & 0\phantom{^{*}} & 4a'
		\end{pmatrix}.
	\end{align*}
	Therefore, the set $ V^{**} $ of its eigenvectors
	\begin{align*}
		v_1^{**} = 
		\renewcommand{\arraystretch}{1.2}
		\begin{pmatrix}
			v_{11}^{**} \\
			v_{12}^{**} \\
			v_{13}^{**} \\
			v_{14}^{**}
		\end{pmatrix}, 
		\ldots, 
		v_4^{**} = 
		\renewcommand{\arraystretch}{1.2}
		\begin{pmatrix}
			v_{41}^{**} \\
			v_{42}^{**} \\
			v_{43}^{**} \\
			v_{44}^{**} 
		\end{pmatrix}, 		
	\end{align*} 
	yields 
	\begin{align*}
		V^{**} = 
		\begin{pmatrix}
			v_1^{**} & v_2^{**} & v_3^{**} & v_4^{**}
		\end{pmatrix} 
		=
		\renewcommand{\arraystretch}{1.2}
		\begin{pmatrix}
			v_{11}^{**} & 0 & 0 & v_{41}^{**} \\
			0 & 0 & v_{32}^{**} & 0 \\
			0 & v_{23}^{**} & 0 & 0 \\
			v_{14}^{**} & 0 & 0 & v_{44}^{**}
		\end{pmatrix},
	\end{align*}
	with $ \lambda_1^{**}, \ldots, \lambda_4^{**} $ as their corresponding eigenvalues.	From the geometric nature of the problem, the dataset centrally symmetric w.r.t. the coordinate origin and axially symmetric w.r.t. both $ x $- and $ y $-axis must necessarily be fitted with axes-aligned, origin-centred conic, therefore, neither of the eigenvectors $ v_2^{**} $ and $ v_3^{**} $ can be the solution, since they do not correspond to an origin-centred conic (also, neither of the eigenvalues $ \lambda_2^{**} $ and $ \lambda_3^{**} $ can represent the optimal solution). Further analysis would yield
	\begin{align*}
		\renewcommand{\arraystretch}{1.6}
		\begin{pmatrix}
			\lambda_1^{**} \\
			\lambda_4^{**}
		\end{pmatrix} 
		= 4
		\begin{pmatrix}
			\lambda_1^{al,0} \\
			\lambda_2^{al,0}
		\end{pmatrix} 
	\end{align*}
	and  
	\begin{align*}
		\setlength\arraycolsep{3pt}
		\begin{pmatrix}
			v_1^{**} & v_4^{**}
		\end{pmatrix} 
		\approx
		\begin{pmatrix}
			v_1^{al,0} & v_2^{al,0}
		\end{pmatrix}.
	\end{align*}
	Therefore, similarly to situation in proof of Proposition~\ref{solution_centred_symmetrisation}, the optimal eigenvalue of the matrix $ P_{con}^{al,0,sym} $ corresponds to the same solution as the one obtained by using the matrix $ P_{con}^{al,0} $ in Proposition~\ref{solution_aligned_centred}, which is also equivalent to the solution acquired in Proposition~\ref{solution_aligned_centred_substitution}. 
\end{proof}

\subsection*{Implementation}
In this subsection we describe three algorithms for fitting an axes-aligned origin-centred conic $ \Qalo $, each of them corresponding to one of the presented Propositions:
\begin{itemize}
	\item[$\bullet$] \textbf{QAL0} --- via eigenvalues and eigenvectors in Proposition~\ref{solution_aligned_centred} \\
	\item[$\bullet$] \textbf{QAL0-dir} --- `direct' computation using Proposition~\ref{solution_aligned_centred_substitution} without computing eigenvalues \\
	\item[$\bullet$] \textbf{QAL0-sym} --- central and axial symmetrisation of the dataset according to Proposition~\ref{solution_aligned_centred_symmetrisation} 
\end{itemize}
As in the case of previous algorithms, we use in Algorithms \textbf{QAL0} and \textbf{QAL0-dir} the reduced forms of vectors and matrices. Therefore, we can define the $ i $-th reduced GAC point
\begin{equation*} 
	\hat{P}_i^{al,0} = 
	\setlength\arraycolsep{4pt}
	\begin{pmatrix}
		1& \frac12(x_i^2+y_i^2)& \frac12(x_i^2-y_i^2)
	\end{pmatrix}^T
\end{equation*}
and create the corresponding reduced data matrix $ \hat{D}^{al,0} $ of the size $ 3 \times N $, where the $ i $-th column is $ \hat{P}_i^{al,0} $. Additionally, by removing the rows 1,4, 5,7 and 8 and columns 1,2, 4, 5 and 8 from the matrix $ B $ of the form \eqref{B}, we also define a reduced matrix for GAC inner product
\begin{align*}
	\hat{B}^{al,0} = 
	\renewcommand{\arraystretch}{1.2}
	\setlength\arraycolsep{3pt}
	\begin{pmatrix}
		\phantom{-}0&\phantom{-}0&-1\phantom{|}\\
		\phantom{-}0&-1&\phantom{-}0\phantom{|}\\
		-1&\phantom{-} 0&\phantom{-}0\phantom{|}
	\end{pmatrix}.
\end{align*} 
Reduction of the matrix $ B_c $ used in the following two algorithms is the matrix $ \Bco $ of the form \eqref{Bco}.
\\
\rule{\linewidth}{1.5pt}
\\
\textbf{\algorithmQALo}
\\
\rule[6pt]{\linewidth}{0.75pt} 
\\
\begin{minipage}{\textwidth}
	\begin{minipage}{.05\textwidth}
		\lstset{
			showlines=true
		}
		\begin{lstlisting}[style=Matlab-editor, rulecolor=\color{white}, mathescape]
$1.$

		\end{lstlisting}
	\end{minipage}
	\hfill
	\begin{minipage}{.95\textwidth}
		\begin{lstlisting}[framexrightmargin=-5.5pt, style=Matlab-editor]
B = [0  0 -1; 0 -1  0; -1  0  0];
Bc = B(2:3,1:2);
		\end{lstlisting}
	\end{minipage}
\end{minipage}
\begin{minipage}{\textwidth}
	\begin{minipage}{.05\textwidth}
		\lstset{
			showlines=true
		}
		\begin{lstlisting}[style=Matlab-editor, rulecolor=\color{white}, mathescape]
$2.$
			
			
		\end{lstlisting}
	\end{minipage}
	\hfill
	\begin{minipage}{.95\textwidth}
		\begin{lstlisting}[framexrightmargin=-5.5pt, style=Matlab-editor]
D = ones(3,N);
D(2,:) = 1/2*(p(1,:).^2 + p(2,:).^2);
D(3,:) = 1/2*(p(1,:).^2 - p(2,:).^2);
		\end{lstlisting}
	\end{minipage}
\end{minipage}
\begin{minipage}{\textwidth}
	\begin{minipage}{.05\textwidth}
		\lstset{
			showlines=true
		}
		\begin{lstlisting}[style=Matlab-editor, rulecolor=\color{white}, mathescape]
$3.$
			
			
			
		\end{lstlisting}
	\end{minipage}
	\hfill
	\begin{minipage}{.95\textwidth}
		\begin{lstlisting}[framexrightmargin=-5.5pt, style=Matlab-editor]
P = 1/N*B*(D*D')*B';
Pc = P(2:3,2:3);
P0 = P(1,1);
P1 = P(1,2:3);
		\end{lstlisting}
	\end{minipage}
\end{minipage}
\begin{minipage}{\textwidth}
	\begin{minipage}{.05\textwidth}
		\lstset{
			showlines=true
		}
		\begin{lstlisting}[style=Matlab-editor, rulecolor=\color{white}, mathescape]
$4.$
			
			
		\end{lstlisting}
	\end{minipage}
	\hfill
	\begin{minipage}{.95\textwidth}
		\begin{lstlisting}[framexrightmargin=-5.5pt, style=Matlab-editor]
Pcon = Bc*(Pc-P1'*1/P0*P1);
[EV,ED] = eig(Pcon);
EW = diag(ED);	
		\end{lstlisting}
	\end{minipage}
\end{minipage}
\begin{minipage}{\textwidth}
	\begin{minipage}{.05\textwidth}
		\lstset{
			showlines=true
		}
		\begin{lstlisting}[style=Matlab-editor, rulecolor=\color{white}, mathescape]
$5.$
			
		\end{lstlisting}
	\end{minipage}
	\hfill
	\begin{minipage}{.95\textwidth}
		\begin{lstlisting}[framexrightmargin=-5.5pt, style=Matlab-editor]
k_opt = find(EW == min(EW(EW>0)));
v_opt = EV(:,k_opt);		
		\end{lstlisting}
	\end{minipage}
\end{minipage}
\begin{minipage}{\textwidth}
	\begin{minipage}{.05\textwidth}
		\lstset{
			showlines=true
		}
		\begin{lstlisting}[style=Matlab-editor, rulecolor=\color{white}, mathescape]
$6.$
			
		\end{lstlisting}
	\end{minipage}
	\hfill
	\begin{minipage}{.95\textwidth}
		\begin{lstlisting}[framexrightmargin=-5.5pt, style=Matlab-editor]
kappa = v_opt'*Bc*v_opt;
v_opt = 1/sqrt(kappa)*v_opt; 
		\end{lstlisting}
	\end{minipage}
\end{minipage}
\begin{minipage}{\textwidth}
	\begin{minipage}{.05\textwidth}
		\lstset{
			showlines=true
		}
		\begin{lstlisting}[style=Matlab-editor, rulecolor=\color{white}, mathescape]
$7.$
			
		\end{lstlisting}
	\end{minipage}
	\hfill
	\begin{minipage}{.95\textwidth}
		\begin{lstlisting}[framexrightmargin=-5.5pt, style=Matlab-editor]
w = -1/P0*P1*v_opt;
Q = [0;w;v_opt(1);0;0;v_opt(2);0;0];
		\end{lstlisting}
	\end{minipage}
\end{minipage}
\\
\rule[12pt]{\linewidth}{1.5pt}
Computation of the matrix $ P $ in step 3. then yields a $ 3 \times 3$ reduced matrix of the form
\begin{align*} 
	\hat{P}^{al,0}=
	\renewcommand{\arraystretch}{1.6}
	\setlength\arraycolsep{2pt}
	\begin{pmatrix}
		P_0^{al} & P_1^{al,0}\\
		{P_1^{al,0}}^T & P_c^{0}\\
	\end{pmatrix}.
\end{align*}
Therefore, after defining a reduced GAC vector of origin-centred conic
\begin{align*}
	\hat{Q}^{al,0}=&
	\setlength\arraycolsep{4pt}
	\begin{pmatrix}
		\vmb & \vpb & \vp 
	\end{pmatrix}^T,
\end{align*}
the objective function \eqref{cost_aligned_centred} can be equivalently formulated as
\begin{align*}
	\hat{Q}^{al,0} \mapsto \hat{Q}^{{al,0}^T} \hat{P}^{al,0} \hat{Q}^{al,0}.
\end{align*}
Next, we can reuse the first steps of \algorithmQALo{} in \algorithmQALodir, while the other steps of algorithm will be altered.

\noindent
\rule{\linewidth}{1.5pt}
\\
\textbf{\algorithmQALodir}
\\
\rule[6pt]{\linewidth}{0.75pt} \vspace*{-21pt}
\begin{enumerate}[leftmargin=1.05cm]
	\item[1.–3.] See \algorithmQALo.
	\vspace{-0.6em}	
	\item[4.] Formation of the matrix $P_{con}^{al,0}$ according to \eqref{Pcon_aligned_centred} and extraction of its elements $ r' $ and $ s' $.
	\begin{lstlisting}[framexrightmargin=-3pt, style=Matlab-editor]
Pcon = Bc*(Pc-P1'*1/P0*P1);
r = Pcon(1,2);
s = Pcon(2,1);
	\end{lstlisting}
	\vspace{-0.6em}
	\item[5.] Direct computation of the optimal vector  $ v^{al,0} $  given by \eqref{valo_substitution}.
	\begin{lstlisting}[framexrightmargin=-3pt, style=Matlab-editor]
vpb = (r/(4*s))^0.25;  
vp = -1/(2*vpb);                      
v_opt = [vpb;vp];
	\end{lstlisting}
	\vspace{-0.6em}		
	\item[6.] Computation of $w$ by \eqref{w_aligned_centred_substitution} and forming the optimal vector $ \Qalo $ according to Proposition~\ref{solution_aligned_centred_substitution}
	\begin{lstlisting}[framexrightmargin=-3pt, style=Matlab-editor]
w = -1/P0*P1*v_opt;
Q = [0;w;vpb;0;0;vp;0;0];
	\end{lstlisting}		
\end{enumerate}
\rule[12pt]{\linewidth}{1.5pt} 
Finally, we summarise a MATLAB code for \algorithmQALosym, which is basically \algorithmQ{} preceded by prior extension of the used dataset. Reduced forms of vectors and matrices used in the following algorithm correspond to the reduced forms described in the commentary to \algorithmQ.
\\
\rule{\linewidth}{1.5pt}
\\
\textbf{\algorithmQALosym}
\\
\rule[6pt]{\linewidth}{0.75pt} \vspace*{-21pt}
\begin{enumerate}[leftmargin=1.05cm]
	\item[0.] Extension of the point dataset by its centrally symmetric images w.r.t. the origin of coordinate system and by points axially symmetric w.r.t. $ x- $ and $ y-$axis.
	\begin{lstlisting}[framexrightmargin=-3pt, style=Matlab-editor]
p = [p,-p];
p = [p, [p(1,:); -p(2,:)]];
N = 4*N;
	\end{lstlisting}
	\vspace{-0.6em}
	\item[1.–2.] See steps 1* and 2* from the commentary to \algorithmQ 
	\vspace{-0.6em}
	\item[3.–7.] See \algorithmQ
\end{enumerate}
\rule[12pt]{\linewidth}{1.5pt} 
Let us note that \algorithmQALosym{} uses four times bigger dataset than the previous two algorithms, so it is not ideal for larger amount of data, at least with regard to computational complexity. Nevertheless, we included it because of its geometrical significance.

\subsection{Overview of fitting methods}

\begin{table}[H] 
	\centering
	\setlength{\tabcolsep}{8pt}
	\caption{Overview of fitting methods}
	\label{table_overview_of_methods}
	\begin{tabular}{@{}cccc@{}}
		\toprule
		Conic                                                                                    & Algorithm         & \begin{tabular}[c]{@{}c@{}}Computational\\ method\end{tabular}                               & \begin{tabular}[c]{@{}c@{}}Objective\\ function\end{tabular} \\ \midrule
		\begin{tabular}[c]{@{}c@{}}general conic\\ $Q$\end{tabular}                        & \textbf{Q}        & \begin{tabular}[c]{@{}c@{}}eigenvalues\\ \&\\ eigenvectors\end{tabular}                      & $Q^T P Q$                                                    \\ \midrule
		\begin{tabular}[c]{@{}c@{}}axes-aligned conic\\ $Q^{al}$\end{tabular}                    & \textbf{QAL}      & \begin{tabular}[c]{@{}c@{}}eigenvalues\\ \&\\ eigenvectors\end{tabular}                      & ${Q^{al}}^T P Q^{al}$                                        \\ \midrule
		& \textbf{Q0}       & \begin{tabular}[c]{@{}c@{}}eigenvalues\\ \&\\ eigenvectors\end{tabular}                      &                                                              \\ \cmidrule(lr){2-3}
		\begin{tabular}[c]{@{}c@{}}origin-centred conic\\ $Q^0$\end{tabular}                     & \textbf{Q0-dir}   & \begin{tabular}[c]{@{}c@{}}direct use of\\ matrix entries\end{tabular}                       & $ {Q^0}^T P Q^0 $                                            \\ \cmidrule(lr){2-3}
		& \textbf{Q0-sym}   & \begin{tabular}[c]{@{}c@{}}central\\ symmetrisation\\ of point dataset\end{tabular}          &                                                              \\ \midrule
		& \textbf{QAL0}     & \begin{tabular}[c]{@{}c@{}}eigenvalues\\ \&\\ eigenvectors\end{tabular}                      &                                                              \\ \cmidrule(lr){2-3}
		\begin{tabular}[c]{@{}c@{}}axes-aligned\\ origin-centred conic\\ $Q^{al,0}$\end{tabular} & \textbf{QAL0-dir} & \begin{tabular}[c]{@{}c@{}}direct use of\\ matrix entries\end{tabular}                       & $ {Q^{al,0}}^T P Q^{al,0}$                                   \\ \cmidrule(lr){2-3}
		& \textbf{QAL0-sym} & \begin{tabular}[c]{@{}c@{}}central \& axial\\ symmetrisation\\ of point dataset\end{tabular} &                                                              \\ \bottomrule
	\end{tabular}
\end{table}

Together with the original \algorithmQ, we described eight conic fitting algorithms in total, each of them fitting one of four types of conics: a general conic~$ Q $, an axes-aligned conic~$ \Qal $, an origin-centred conic~$ \Qo $ or an axes-aligned origin-centred conic~$ \Qalo $. Overview of the fitting methods by a conic type with the corresponding algorithms, types of computational methods and the objective functions is given by Table~\ref{table_overview_of_methods}.

Let us point out that in the algorithms using the eigenvalues and eigenvectors of a certain matrix operator, the calculation of the corresponding operator has essentially the same structure for all the algorithms. As the default equation we take \eqref{Pcon}, where the matrix $ P_{con} $---the operator used in \algorithmQ{}---is computed using the central subblock $ B_c $ of the matrix $ B $ and the subblocks $ P_0 $, $ P_1 $ and $ P_c $ of matrix $ P $ given by the decomposition \eqref{Pblocks}. Consequently, the matrix operators used in Algorithms~\textbf{QAL}, \textbf{Q0} and \textbf{QALO} are, indeed, calculated with equations almost identical to equation \eqref{Pcon}, only the subblock of the matrix $ B $ and the subblocks of matrix $ P $ are generally substituted by different subblocks, according to the corresponding decomposition of matrices $ B $ and $ P $. The overview of the matrix operators computation by a conic type is given by Table~\ref{table_matrix_operator}.
\begin{table}[h]
	\centering
	\setlength{\tabcolsep}{4pt}
	\caption{Construction of matrix operator for eigenvalue problem by conic type}
	\label{table_matrix_operator}
	\begin{tabular}{@{}cccccc@{}}
		\toprule
		\multirow{2}{*}{Conic}                                                                   & \begin{tabular}[c]{@{}c@{}}subblock \\ of $ B $\end{tabular} & \multicolumn{3}{c}{\begin{tabular}[c]{@{}c@{}}subblocks \\ of $ P $\end{tabular}} & \multirow{2}{*}{Matrix operator}                                          \\ \cmidrule(lr){2-5}
		& $B_c^*$                                                 & $P_0^*$                  & $P_1^*$                    & $P_c^*$              &                                                                           \\ \midrule
		\begin{tabular}[c]{@{}c@{}}general conic \\ $Q$\end{tabular}                             & $B_c$                                                   & $P_0$                    & $P_1$                      & $P_c$                & \multirow{4}{*}{\vspace{-60pt} $P_{con}^* = B_c^* (P_c^* - {P_1^*}^T {P_0^*}^{-1} P_1^*)$} \\ \cmidrule(r){1-5}
		\begin{tabular}[c]{@{}c@{}}axes-aligned conic\\ $Q^{al}$\end{tabular}                    & $B_c$                                                   & $P_0^{al}$               & $P_1^{al}$                 & $P_c$                &                                                                           \\ \cmidrule(r){1-5}
		\begin{tabular}[c]{@{}c@{}}origin-centred conic\\ $Q^0$\end{tabular}                     & $B_c^0$                                                 & $P_0$                    & $P_1^0$                    & $P_c^0$              &                                                                           \\ \cmidrule(r){1-5}
		\begin{tabular}[c]{@{}c@{}}axes-aligned\\ origin-centred conic\\ $Q^{al,0}$\end{tabular} & $B_c^0$                                                 & $P_0^{al}$               & $P_1^{al,0}$               & $P_c^0$              &                                                                           \\ \bottomrule
	\end{tabular}
\end{table}

\section{Experimental results and applications}\label{sect5}
Finally, we apply each of the presented algorithms (including the original \algorithmQ) on a sample dataset of 10 points described by Table~\ref{table_points}. After doing so, we compare the resulting conics both visually and by values of the corresponding objective functions. 
\begin{table}[h] 
	\centering
	\setlength{\tabcolsep}{8pt}
	\caption{Fitted points}
	\label{table_points}
	\begin{tabular}{@{}ccccccccccc@{}}
		\toprule
		$x_i$ & $ -3 $ & $ -3 $ & $ -2 $ & $ -3 $ & $ -6 $ & $ -7 $ & $ -9 $ & $ -10 $ & $ -9 $ & $ -7 $ \\ \midrule
		$y_i$ & \phantom{$-$}3  & \phantom{$-$}4  & \phantom{$-$}6  & \phantom{$-$}7  & \phantom{$-$}8  & \phantom{$-$}8  & \phantom{$-$}7  & \phantom{$-$}4   & \phantom{$-$}2  & \phantom{$-$}1  \\ \bottomrule
	\end{tabular}
\end{table}
Resulting conics sorted by type and used algorithm can be seen in Figures~1–4. Even though all three algorithms for fitting an origin-centred conic $ \Qo $ produce the same result, an additional subfigure (b) was made in Figure~3 to emphasise the fact that \algorithmQosym{} uses symmetrically extended dataset, unlike the cases of Algorithms \textbf{Q0} and \textbf{Q0-dir}. Two subfigures in Figure~4 have the same purpose in relation to axes-aligned origin-centred conic $ \Qalo $.
\begin{figure}[p!]
		\begin{subfigure}[b]{\linewidth}
			\begin{subfigure}[b]{0.45\linewidth}
				\centering
				\includegraphics[width=\linewidth]{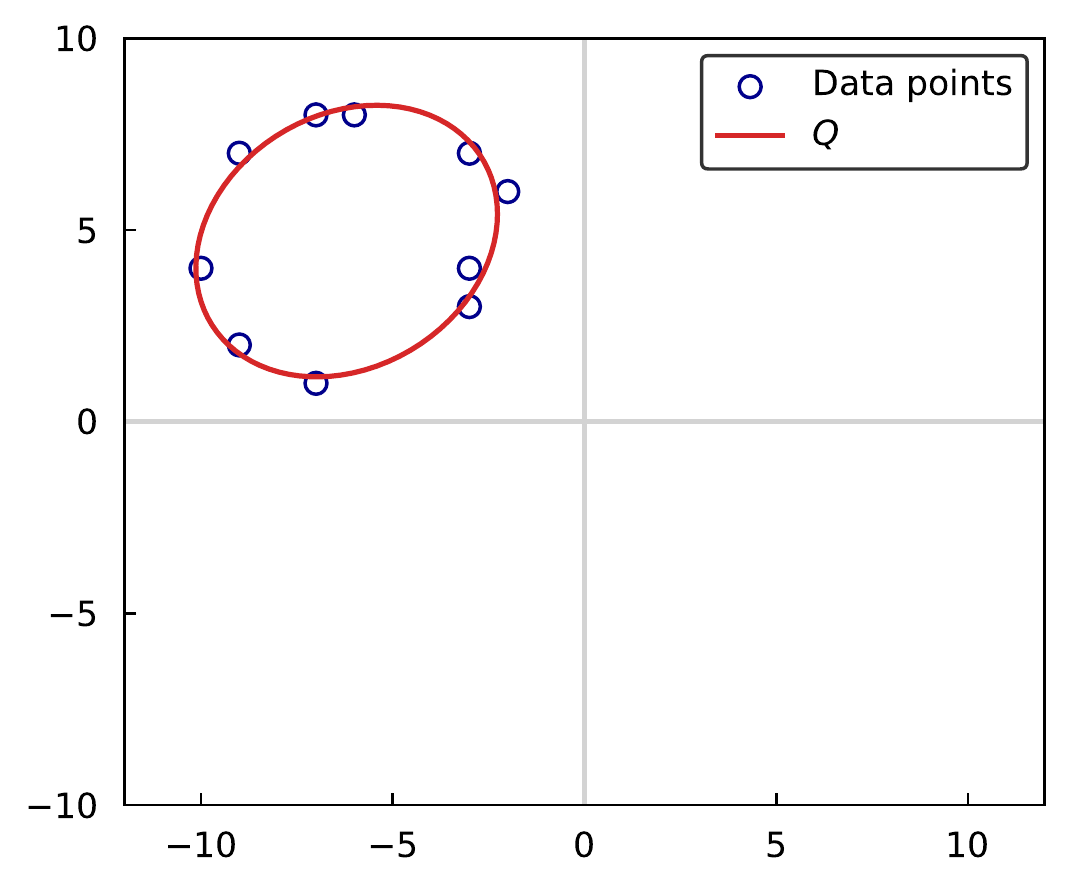}
				\caption*{\textbf{Q}}
				\caption*{\textsc{Fig. 1:} general conic $ Q $}
			\end{subfigure}
			\hfill
			\begin{subfigure}[b]{0.45\linewidth}
				\centering
				\includegraphics[width=\linewidth]{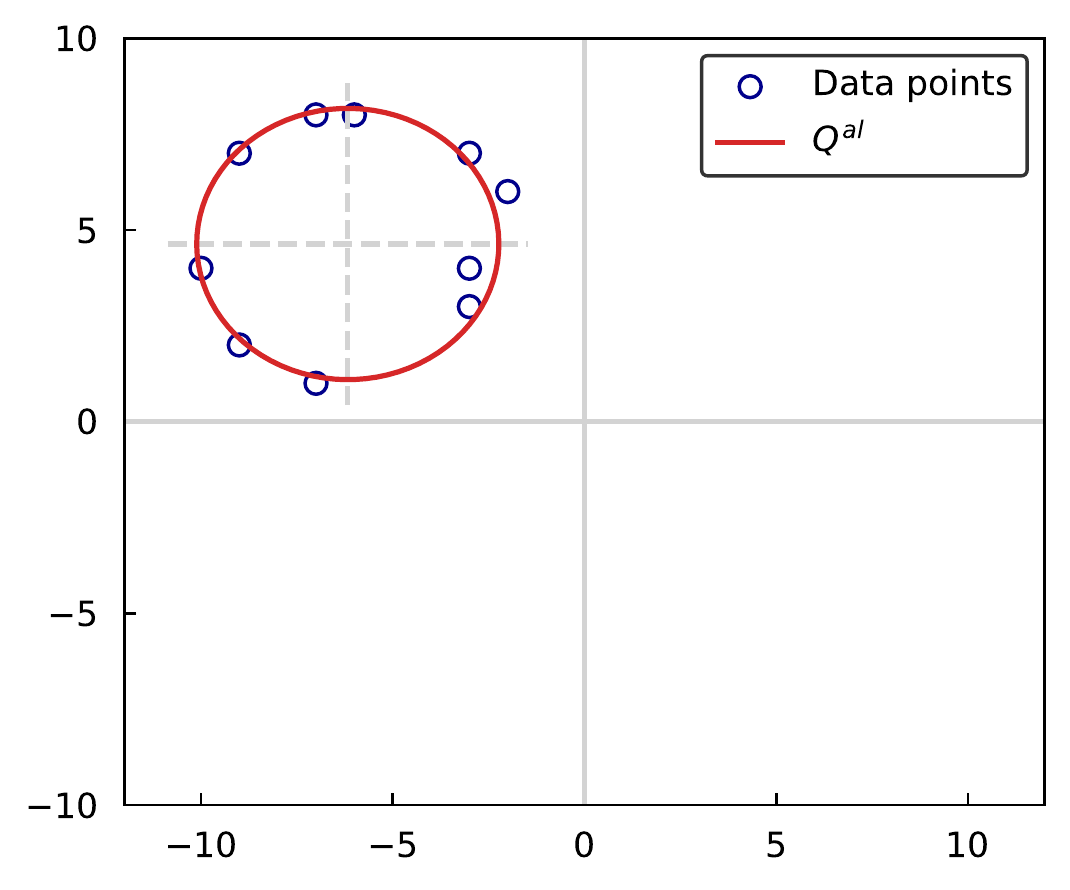}
				\caption*{\textbf{QAL}}
				\caption*{\textsc{Fig. 2:} axes-aligned conic $ \Qal $}
			\end{subfigure}
		\end{subfigure}
		\par\bigskip 	
		\begin{subfigure}[b]{\linewidth}
			\begin{subfigure}[b]{0.45\linewidth}
				\centering
				\includegraphics[width=\linewidth]{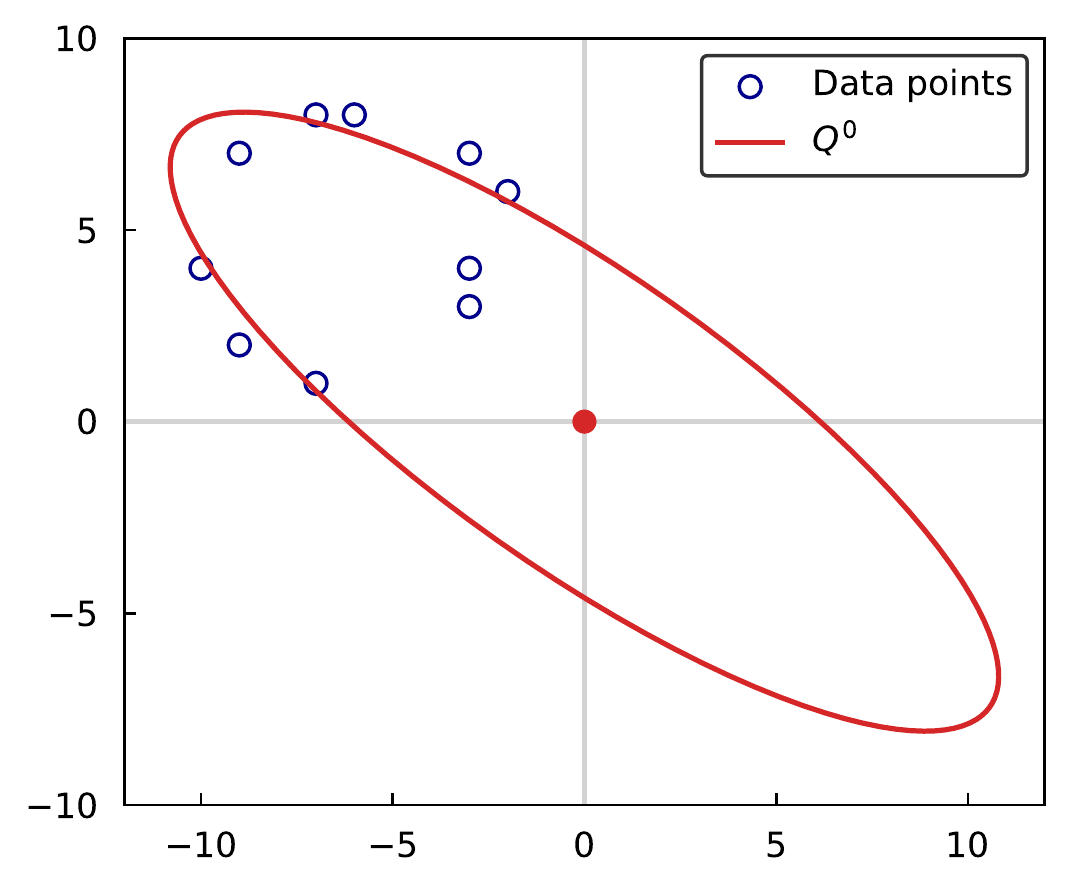}
				\caption{\textbf{Q0}, \textbf{Q0-dir}}
			\end{subfigure}
			\hfill
			\begin{subfigure}[b]{0.45\linewidth}
				\centering
				\includegraphics[width=\linewidth]{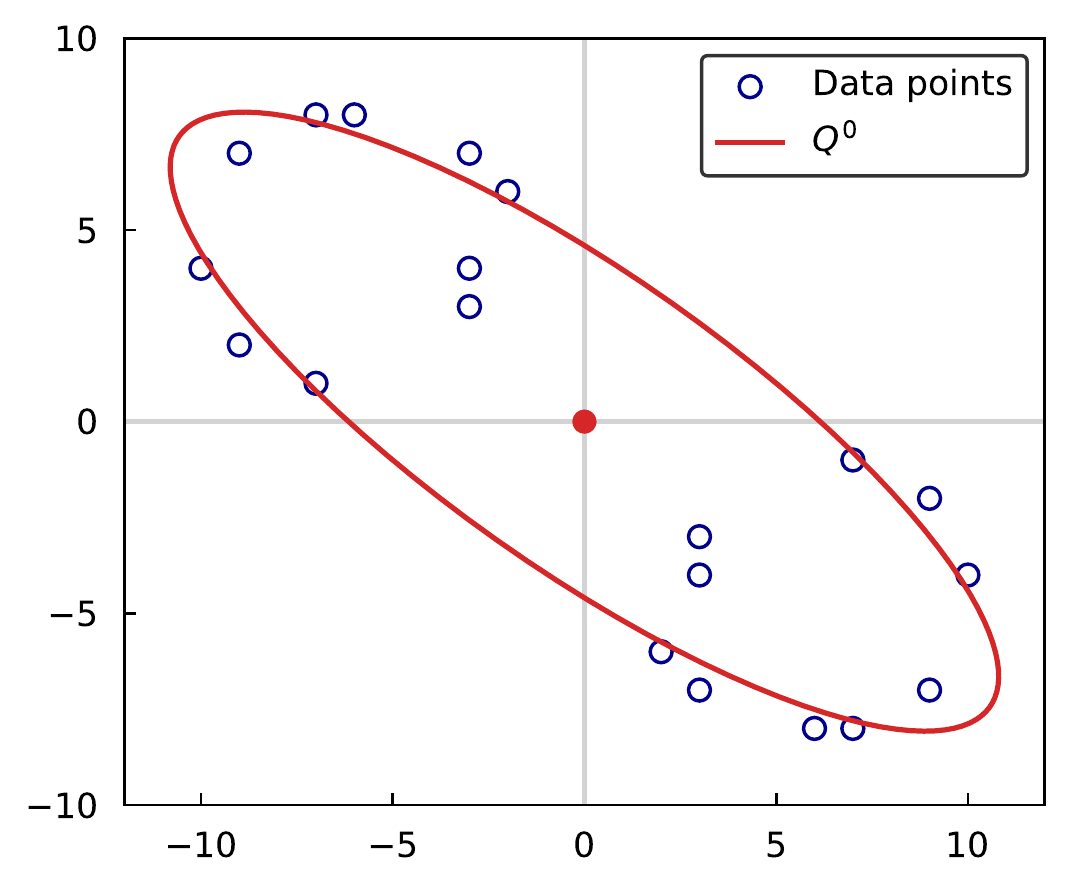}
				\caption{\textbf{Q0-sym}}
			\end{subfigure}
			\vspace{4pt}
			\caption*{\textsc{Fig. 3:} origin-centred conic $ \Qo $}
		\end{subfigure}	
		\par\bigskip 
		\begin{subfigure}[b]{\linewidth}
			\begin{subfigure}[b]{0.45\linewidth}
				\centering
				\includegraphics[width=\linewidth]{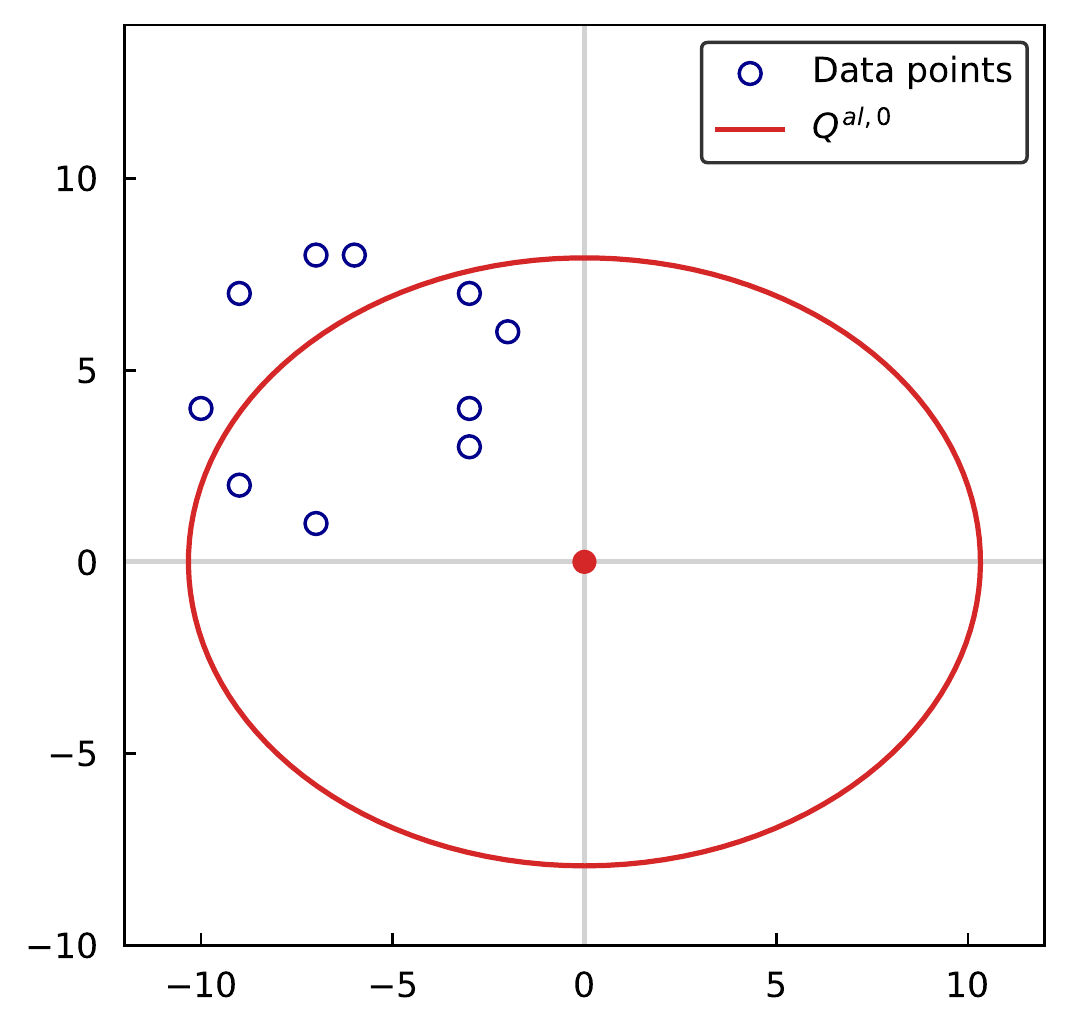}
				\caption*{(a) \textbf{QAL0}, \textbf{QAL0-dir}}
			\end{subfigure}
			\hfill
			\begin{subfigure}[b]{0.45\linewidth}
				\centering
				\includegraphics[width=\linewidth]{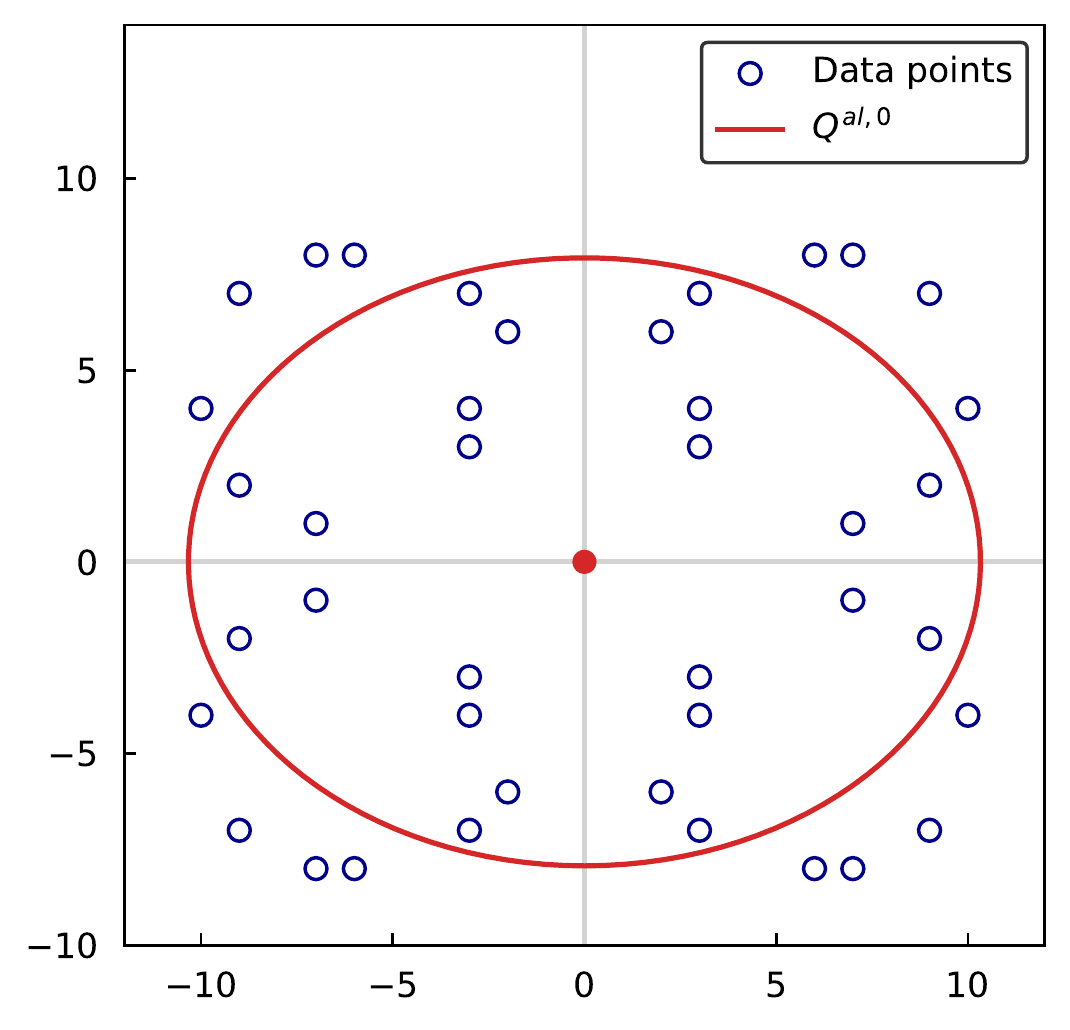}
				\caption*{(b) \textbf{QAL0-sym}}
			\end{subfigure}
			\vspace{6pt}
			\caption*{\textsc{Fig. 4:} axes-aligned origin-centred conic $ \Qalo $}
		\end{subfigure}
\end{figure}
It is obvious that in our case the new conic fitting methods with additional geometric constraints produce conics with greater total distance from the data points than general conic $ Q $ acquired by \algorithmQ. Even though the difference is not that evident when comparing conics $ Q $ and $ \Qal $, one cannot overlook the increase of total distance from the data points to conic $ \Qo $, let alone the case of conic $ \Qalo $. Consequently, we can conclude that the ability to achieve a tight fit is decreasing with the fitted conic's degree of freedom reduction (at least for our dataset), in other words, value of objective function is inversely proportional to the degrees of freedom (d.f.). Indeed, we can see this trend in Figures~1–4: a general conic $ Q $ has five d.f. in total; an axes-aligned conic $ \Qal $ has one less d.f. due to the fixed tilt of axes; an origin-centred conic $ \Qo $ loses two d.f. because of the fixed position of its centre-point; finally, a conic $ \Qalo $ has three d.f. less than a general conic $ Q $, since it has both the tilt of axes and the centre-point position fixed. The fits depicted in Figures~1–4 are characterised by the corresponding conic type, value of the objective function and degrees of freedom in Table~\ref{table_obj_function_freedom}. 
\begin{table}[h]
	\centering
	\setlength{\tabcolsep}{8pt}
	\caption{Objective function \& geometric freedom}
	\label{table_obj_function_freedom}
	\begin{tabular}{@{}ccc@{}}
		\toprule
		Conic        & value of objective function & degrees of freedom \\ \midrule
		$ Q $        & 0.0245                      & 5                  \\
		$ \phantom{^{al}}Q^{al} $   & 0.0871                      & 4                  \\
		$ \phantom{^{0}}Q^{0} $    & 1.0369                      & 3                  \\
		$ \phantom{^{al,0}}Q^{al,0} $ & 4.3361                      & 2                  \\ \bottomrule
	\end{tabular}
\end{table} 
Due to the layout of the points in the used dataset, conics $ \Qo $ and $ \Qalo $ are not exactly tightly fitted, nonetheless, such fits can be conveniently used for data consisting of points that approximately lie on an origin-centred, or axes-aligned origin-centred conic, respectively. One of the fields, where the mentioned fits can be used, are so called \textit{switched systems}, with conics representing the integral curves (trajectories) of underlying system of ODE's. At the beginning, the system of ODE's is numerically solved (e.g. by Runge-Kutta method) with the initial condition at the starting point $ A $. This will give us a set of points representing the initial conic. Subsequently, the GAC conic fitting algorithm is applied and thus we get the conic in IPNS representation. In some cases, we know beforehand that the resulting integral curves of ODE's should be origin-centred conics, therefore, it is more appropriate to fit the points with one of Algorithms \textbf{Q0}, \textbf{Q0-dir} and \textbf{Q0-sym} instead of \algorithmQ{}, which cannot ensure additional geometric constraints. Moreover, other ODE's lead to trajectories simultaneously axes-aligned and origin-centred, so any of Algorithms \textbf{QAL0}, \textbf{QAL0-dir} and \textbf{QAL0-sym} can be applied in that case. Such an approach makes the initial trajectories and the following GAC calculations very precise. For more details on the topic, see \cite{DV}.

\section{Conclusion}
At the beginning of the text, we briefly summarised the concept of GAC and recalled the original conic fitting algorithm using GAC and specific normalisation constraint. Consequently, by using basic knowledge of conic sections theory and exploiting the matrix formulation of the original fitting problem, \cite{GACFIT}, we derived methods of fitting conics with one of three additional geometric properties, in particular: having the axes aligned with the coordinate axes, having the centre located at the origin of coordinates, and, eventually, two previous properties combined. 

Each of the presented fitting methods is followed by the corresponding MATLAB algorithm (or multiple algorithms with equivalent output), so, together with the original algorithm, we presented eight conic fitting algorithms in total, each resulting in one of four geometrically (un)constrained types of conics. Moreover, redundant elements used in the derivations of the fitting methods were mostly omitted in the described algorithms (including the original algorithm), thus the computational demands on both time and memory have been reduced.

All the methods were tested on a sample dataset and the resulting conics were analysed visually and by the values of the corresponding objective functions. In particular, our empirical results showed a relationship between a fitted conic's degrees of freedom and tightness of the fit to the data points: the more degrees of freedom, the tighter the fit. We point out that all GAC fitting algorithms are not invariant under translation, which is a natural property due to the structure of GAC. Translation invariant algorithm is a subject of further research.

Finally, we mentioned possible use of the algorithms in applications.

\section*{Acknowledgement}
The research was supported by a grant no. FSI-S-20-6187.

\newpage
\noindent
PAVEL LOU\v CKA, PETR VA\v S\'IK 
\newline
Institute of Mathematics
\newline 
Brno University of Technology
\newline
Faculty of Mechanical Engineering, 
\newline
Technick\'a 2
\newline
616 69 Brno, Czech Republic
\newline
\noindent
e-mail:{\tt \ Pavel.Loucka@vutbr.cz, Petr.Vasik@vutbr.cz }

\end{document}